\documentclass[12pt,twoside,a4paper]{amsart}
\usepackage{amssymb}
\usepackage{amsmath,amscd}
\usepackage[initials,nobysame,shortalphabetic]{amsrefs}
\usepackage[all,cmtip]{xy}
\usepackage[hidelinks]{hyperref}

\usepackage[utf8]{inputenc}

\def\End{{\rm End}}

\def\dbar{\bar\partial}

\def\C{{\mathbb C}}

\def\Cn{\C^n}

\def\D{{\mathcal D}}

\def\PM{{\mathcal{PM}}}

\def\codim{{\rm codim\,}}

\def\Im{{\rm Im\, }}

\def\Ok{{\mathcal O}}

\def\Re{{\rm Re\,  }}

\def\reg{{\rm reg}}

\newcommand{\Com}[1]{}
\DeclareMathOperator{\Id}{Id}
\DeclareMathOperator{\supp}{supp}
\DeclareMathOperator{\ann}{ann}

\DeclareMathOperator{\im}{im}

\DeclareMathOperator{\Ass}{Ass}

\def\be{\begin{equation}}
\def\ee{\end{equation}}

\newtheorem{thm}{Theorem}[section]
\newtheorem{lma}[thm]{Lemma}

\newtheorem{prop}[thm]{Proposition}

\theoremstyle{definition}

\newtheorem{df}{Definition}

\theoremstyle{remark}

\newtheorem{preremark}{Remark}
\newtheorem{preex}{Example}

\newenvironment{ex}{\begin{preex}}{\end{preex}}

\numberwithin{equation}{section}

\begin{document}

\title[Residue currents on singular varieties]{Residue currents with prescribed annihilator ideals on singular varieties}

\date{\today}

\author{Richard L\"ark\"ang}

\address{R. L\"ark\"ang\\School of Mathematical Sciences, University of Adelaide, Adelaide SA 5005, Australia}

\email{richard.larkang@adelaide.edu.au}

\subjclass{32A27, 32C30}

\keywords{residue currents, local duality, singular varieties, local analytic geometry}

\begin{abstract}
    Given an ideal $\mathcal{J}$ on a complex manifold, Andersson and Wulcan
    constructed explicitly a current $R^\mathcal{J}$ such that the annihilator of $R^\mathcal{J}$ is $\mathcal{J}$,
    generalizing the duality theorem for Coleff-Herrera products.
    We describe a way to generalize this construction to ideals on singular varieties.
\end{abstract}

\maketitle

\section{Introduction}

Let $f \in \Ok$ be a germ of a holomorphic function, where $\Ok = \Ok_{\Cn,0}$
is the ring of germs of holomorphic functions at the origin in $\Cn$. Consider the problem of
finding a current $U$ such that $f U = 1$. Such currents were proven to exist
abstractly by Schwartz in \cite{Schw}. A canonical and explicit
choice of such a current, as constructed in \cite{HL}, is the \emph{principal value current}
$1/f$, which can be defined by
\begin{equation*}
    \frac{1}{f} := \lim_{\epsilon \to 0^+} \frac{\bar{f}}{|f|^2 + \epsilon},
\end{equation*}
where the limit is taken in the sense of currents. The existence of this limit over
$Z(f)$ as a current is non-trivial if $n > 1$, relying
on Hironaka's theorem on resolution of singularities. Nevertheless, $1/f$ exists as a
explicit limit of smooth functions. In addition, it is canonical in the sense
that any ``reasonable'' way of cutting off the singularities followed by
a limiting procedure will result in the same current.

Since we have defined the principal value current $1/f$, one can also give meaning to
meromorphic currents $g/f$ and residue currents $\dbar(1/f)$.
The residue current $\dbar(1/f)$ is closely related to the ideal $\mathcal{J}(f)$ generated by
$f$ in the following way:
Let $\ann_{\Ok} \dbar(1/f)$ be the \emph{annihilator} of $\dbar(1/f)$, i.e., the ideal of
holomorphic functions $g$ such that $g \dbar(1/f) = 0$. Then $g \in \ann_\Ok \dbar(1/f) = 0$
if and only if $\dbar(g/f) = 0$ and, by regularity of the $\dbar$-operator on
$(0,0)$-currents, this holds if and only if $g/f \in \Ok$, i.e.,
$g \in \mathcal{J}(f)$. Hence, $\ann_\Ok \dbar(1/f) = \mathcal{J}(f)$.

Let $f = (f_1,\dots,f_p) \in \Ok^{\oplus p}$ be a tuple of holomorphic functions.
In \cite{CH}, Coleff and Herrera showed that one can give a meaning to products
\begin{equation*}
    \dbar \frac{1}{f_p}\wedge \dots \wedge \dbar\frac{1}{f_1},
\end{equation*}
what is nowadays called the \emph{Coleff-Herrera product} of $f$, and which we will
also denote by $\mu^f$. 

Such products are ``nicely'' behaved if $f$ defines a \emph{complete intersection},
i.e., if $\codim Z(f) = p$. Maybe the most important property is the following
\emph{duality theorem} for Coleff-Herrera products.
\begin{thm}
    Let $f = (f_1,\dots,f_p)$ be a holomorphic mapping on a complex manifold defining a complete intersection.
    Then locally,
    \begin{equation*}
        \ann \mu^f = \mathcal{J}(f_1,\dots,f_p).
    \end{equation*}
\end{thm}
This result thus extends the description of the annihilator for one single holomorphic
function described above. It was proven independently by Dickenstein
and Sessa in \cite{DS} and Passare in \cite{PMScand}.

Another way in which the Coleff-Herrera product is nicely behaved in the case of complete intersection
is the following. Let $f = (f_1,\dots,f_p)$ and $g = (g_1,\dots,g_p)$ be two tuples of holomorphic
functions defining complete intersections. If there exists a matrix $A$ of holomorphic functions such
that $f = g A$, then the \emph{transformation law} for Coleff-Herrera products states that
$\mu^g = (\det A) \mu^f$. In particular, if $f$ and $g$ define the same ideal, then
$A$ is invertible, so $\det A$ is a non-vanishing holomorphic function. Thus, we can view the
Coleff-Herrera product as an essentially canonical current associated to a complete intersection ideal.

Coleff-Herrera products have had various applications, for example to explicit versions of the
Ehrenpreis-Palamodov fundamental principle by Berndtsson and Passare, \cite{BePa}, the $\dbar$-equation
on singular varieties by Henkin and Polyakov, \cite{HePo}, and effectivity questions in division
problems by Berenstein and Yger, \cite{BeYg}.

In \cite{AW1}, Andersson and Wulcan generalized the construction of the Coleff-Herrera product
from complete intersection ideals to arbitrary ideals. From a Hermitian resolution $(E,\varphi)$
(i.e., a locally free resolution equipped with Hermitian metrics)
of an ideal $\mathcal{J}$, they constructed explicitly a vector-valued current $R^\mathcal{J}$
with values in $E$ such that $\ann_\Ok R^\mathcal{J} = \mathcal{J}$. In case
$\mathcal{J} = \mathcal{J}(f_1,\dots,f_p)$ is a complete intersection ideal,
the current they constructed coincides with the Coleff-Herrera product of $f$.

In case the ideal is Cohen-Macaulay, i.e., if $\Ok/\mathcal{J}$ has a free resolution of length equal
to $\codim Z(\mathcal{J})$, the current $R^\mathcal{J}$ is essentially canonically associated to
$\mathcal{J}$, in the sense that it does not depend on the Hermitian metrics chosen, and choosing
different minimal free resolutions only changes the current by an invertible holomorphic matrix
(just like the Coleff-Herrera product changes by an invertible holomorphic function by changing
the generators).
In addition, the construction ``globalizes'' in the same way as free resolutions in the sense that
if we construct the current $R^\mathcal{J}$ globally, and restrict it to a neighbourhood
of a point $z$, we can express $R^\mathcal{J}$ there as a smooth matrix times the current constructed
locally around $z$ (just as considering a global (locally) free resolution will in general not
restrict to a minimal free resolution locally, but only that the local minimal free resolution
is a direct summand of the restriction of the global one).

The construction is explicit both in the sense that it is explicitly described
in terms of a free resolution of the ideal, and also in the sense that it not
only describes ideal membership in terms of its annihilator, but also explicitly
realizes this ideal membership, by appearing in integral representation formulas,
see \cite{AW1}, Section~5.

The applications described for Coleff-Herrera products have been generalized in
various ways to Andersson-Wulcan currents, thereby being able to remove assumptions
about complete intersection, see for example \cites{AS2,AS3,ASS,AW1,AW3,Sz}.

The aim of this article, is to generalize the construction in \cite{AW1}, to currents with
prescribed annihilator ideals on singular varieties. Describing this construction more precisely,
and how the construction generalizes the one of Andersson and Wulcan requires more knowledge
about their construction, which we leave for later parts of the article,
see in particular Theorem~\ref{thmmain1} and Theorem~\ref{thmmain2}.
In the introduction, we instead describe a special case where many of the
technicalities of the construction disappears, while it still illustrates
much of the ideas behind the construction.

\subsection{Principal ideals on hypersurfaces} \label{ssectpihyper}

Let $Z \subseteq \Omega$ be a \emph{reduced hypersurface} of an open set $\Omega \subseteq \Cn$,
i.e., $Z = Z(h)$, where $h$ is a holomorphic function on $\Omega$ such that $dh$ is
non-vanishing generically on $Z$.
In particular, $\Ok_Z = \Ok/\mathcal{J}(h)$.

One of the simplest examples of an ideal in $\Ok_Z$ would be a principal ideal
$\mathcal{J} = \mathcal{J}(f) \subseteq \Ok_Z$, where we
also assume that $f$ is a non-zero-divisor in $\Ok_Z$, i.e., $f$ does
not vanish identically on any irreducible component of $Z$.
We then want to find an intrinsic current $R$ on $Z$ such that
$\ann_{\Ok_Z} R = \mathcal{J}$. Currents on analytic varieties can either
be defined in a similar manner as on manifolds, or in terms of currents
in the embedding, see Section~\ref{ssectcurrvar}.
Of particular importance here will be that the construction of principal-value
currents works just as well on singular varieties.
Since the residue current $\dbar (1/f)$ of $f$ exists on $Z$,
it would be a natural candidate for the current $R$.
However, in \cite{Lar2}, we show that if
$\codim Z_{\rm sing} = 1$ (as would be the case for example for any singular
planar curve), then one can always find a holomorphic function $f$
such that $\ann_{\Ok_Z} \dbar(1/f) \neq \mathcal{J}(f)$.

We instead start by considering currents in the ambient space. Let
$\tilde{f}$ be a representative of $f$ in the ambient space $\Omega$.
The current
\begin{equation*}
    T := \dbar\frac{1}{\tilde{f}}\wedge\dbar\frac{1}{h}\wedge dz,
\end{equation*}
where $dz = dz_1\wedge\dots\wedge dz_n$, has the same annihilator as
$\dbar(1/\tilde{f})\wedge\dbar(1/h)$, i.e., $\mathcal{J}(\tilde{f},h)$
by the duality theorem.
Since the annihilator contains $h$, we get a well-defined multiplication with elements of
$\Ok_Z = \Ok/\mathcal{J}(h)$, and the annihilator of $T$ over $\Ok_Z$ equals $\mathcal{J}(f)$.
Thus, we have found a current in the ambient space with the correct annihilator, and then
if we can find a current $R$ on $Z$ such that $i_* R = T$, where $i : Z \to \Omega$ is
the inclusion, then $R$ will be a current with the correct annihilator.

We consider the current $(1/f)\omega$ on $Z$, where $\omega$ is the Poincar\'e residue of $dz/h$,
see Example~\ref{expresidue} below.
One way of characterizing the Poincar\'e residue $\omega$ is that $i_* \omega = \dbar(1/h)\wedge dz$, so
\begin{equation*}
    i_*\left(\frac{1}{f}\omega\right) = \frac{1}{\tilde{f}}\dbar\frac{1}{h}\wedge dz.
\end{equation*}
Thus, by Leibniz' rule, see \eqref{eqprodprops},
\begin{equation*}
    i_*\left(\dbar\left(\frac{1}{f}\omega\right)\right)
    = \dbar\left(\frac{1}{\tilde{f}}\dbar\frac{1}{h}\wedge dz\right)
    = \dbar\frac{1}{\tilde{f}}\wedge\dbar\frac{1}{h}\wedge dz = T,
\end{equation*}
and we have proved the following.
\begin{prop} \label{propprincipalonhyper}
    Let $Z$ be a reduced hypersurface defined by a holomorphic function $h$,
    and let $\omega$ be the Poincar\'e residue of $dz/h$ on $Z$.
    If $f \in \Ok_Z$ is a non-zero-divisor and $R^f_Z$ is the current $\dbar( (1/f)\omega)$
    on $Z$, then
    \begin{equation*}
        \ann_{\Ok_Z} R^f_Z = \mathcal{J}(f).
    \end{equation*}
\end{prop}

Note that, since $\dbar \omega = 0$, we have formally that $R^f_Z = \dbar(1/f)\wedge\omega$.
However, it might very well happen that $\omega$ has its poles (which are contained in $Z_{\rm sing}$)
on $Z(f) = \supp \dbar(1/f)$.
In that case, the product $\dbar(1/f)\wedge \omega$ can not be defined in a
``robust'' way. For example, it is natural to regularize the factors one at a time, and in
that case, the product will in general depend on in which order one regularizes,
so we refrain from giving such products any meaning.
However, in case $\codim Z_{\rm sing}\cap Z(f) \geq 2$ in $Z$, then
$\dbar(1/f)\wedge \omega$ can be defined in a ``robust'' way, and it coincides
with $R^f_Z$.

If we let $U = (1/f)$, then by Leibniz' rule,
and by a natural cancellation property for residue currents, see \eqref{eqprodprops}, we get that
\begin{equation} \label{eqrfhomotopy}
    R^f_Z = \omega - \nabla(U\omega),
\end{equation}
where $\nabla = f - \dbar$, and in addition,
\begin{equation} \label{eqrfpushforward}
    i_* R^f_Z = \dbar \frac{1}{\tilde{f}} \wedge \dbar\frac{1}{h}\wedge dz.
\end{equation}
In this article, we generalize this construction to arbitrary ideals on arbitrary varieties.
The starting point of generalizing this construction is to replace the right-hand side of
\eqref{eqrfpushforward} with the Andersson-Wulcan current $R^\mathcal{\tilde{J}}$ associated to a
maximal lifting $\mathcal{\tilde{J}}$ of the ideal $\mathcal{J}$, which will give a current in the
ambient space with the correct annihilator. In Section~\ref{sectprel}, we describe
the construction of residue currents from \cite{AW1} and other necessary background on residue currents.
In order to prove that this current corresponds to a current on $Z$, we show that $R^{\mathcal{\tilde{J}}}\wedge dz$
is the push-forward of a current on $Z$ of a similar form as the right-hand side of \eqref{eqrfhomotopy}.
We treat the case when $Z$ is of pure dimension in Section~\ref{sectpure}. The main ingredients are
a comparison formula for Andersson-Wulcan currents from \cite{Lar3}, relating such currents
associated to two different ideals, and a generalization of the Poincar\'e residue to arbitrary varieties
of pure dimension, as introduced in \cite{AS2}, called the structure form associated to $Z$.
In Section~\ref{sectsmooth}, we describe how this construction coincides with the construction
in \cite{AW1} in case $Z$ is non-singular. In Section~\ref{sectnonpure}, we prove the
general case of our construction, i.e., when $Z$ is not necessarily of pure dimension. A key part
is to prove the existence of a structure form also associated to such varieties.
We finish in Section~\ref{sectnaive} by discussing why a more straightforward generalization of
the construction in \cite{AW1}, by considering free resolutions on the variety itself,
does not work in general.

\section*{Acknowledgements}

I would like to thank Mats Andersson for valuable
discussions in the preparation of this article.

\section{Preliminaries} \label{sectprel}

In this section we recall several tools which will be useful during
the rest of the article, like currents on singular varieties, almost semi-meromorphic
and pseudomeromorphic currents, the construction of Andersson-Wulcan of currents
with prescribed annihilator ideals and a comparison formula for such currents.

\subsection{Currents on analytic varieties} \label{ssectcurrvar}

Since a key part in this article is that we construct intrinsic currents
on the varieties, we begin by recalling what currents on analytic varieties are.
The usual way to define currents on an analytic variety is to first define test forms on
analytic varieties, and then define currents as continuous linear functionals on the
test forms. However, it can also be described more concretely in terms of embeddings.
If $Z$ is a subvariety of pure codimension $k$ of some complex manifold $X$,
and $i$ is the inclusion $i : Z \to X$, then $T$ is a $(p,q)$-current on $Z$ if
$i_* T$ is a $(p+k,q+k)$-current on $X$ which vanishes when acting on
test forms $\phi$ on $X$ such that $\phi|_{Z_\reg} = 0$.
Conversely, if $T'$ is any such current on $X$, then $T'$ defines a unique current
$T$ on $Z$ such that $i_* T = T'$.
Note that considered as a current in the ambient space, it is not sufficient
that $\supp T \subseteq Z$ for it to correspond to a current on $Z$.
For example, if $Z = \{ 0 \} \subseteq \C$, then $[0]$, the integration current at $\{0\}$,
corresponds to a current on $Z$, while $\partial/\partial_z [0]$ does not, although both
have support on $Z$.

\begin{ex}
    The most basic example of a current on a singular variety is given by
    the integration current constructed by Lelong, \cite{Lel}.
    Given a subvariety $Z$ of a complex manifold $X$, the integration current
    $[Z]$ of $Z$ on $X$ is defined by
    \begin{equation*}
        [Z].\phi := \int_{Z_{\rm reg}} \phi,
    \end{equation*}
    where $\phi$ is a test form. It is thus immediate from the description
    above, that $[Z]$ corresponds to a current on $Z$, and it is reasonable
    to denote it by $1$, i.e., $i_* 1 = [Z]$.
\end{ex}

Multiplying the equation $i_* 1 = [Z]$ by a smooth form,
any smooth $(p,q)$-form on $Z$ can be considered as a current on $Z$,
and in fact, the construction of Herrera and Lieberman of principal value and
residue currents works also on a singular variety, so for any meromorphic
$(p,q)$-form $\eta$ on $Z$, we can define its corresponding meromorphic current,
which we for simplicity will also denote by $\eta$.

By a \emph{holomorphic form} on a singular variety $Z$, we mean the restriction of
a holomorphic form in the ambient space, and by a \emph{meromorphic form},
we mean the restriction of a meromorphic form in the ambient space
such that its polar set has positive codimension in $Z$. See
\cite{HP} for a rather detailed discussion about different definitions of meromorphic
forms, and various definitions of holomorphic forms.
In order to distinguish between a meromorphic form $\eta$ on $Z$,
and a representative of it in the ambient space, we will denote the
representative by $\tilde{\eta}$. In particular, we write $i_* \eta = \tilde{\eta} \wedge [Z]$.

In case we have two holomorphic functions $f$ and $g$ on $Z$ such that
$\codim Z(f)\cap Z(g) = 2$, then we can form products of residue currents
and principal value currents of $f$ and $g$ satisfying the following natural
properties.
\begin{equation} \label{eqprodprops}
    f \frac{1}{f}\dbar \frac{1}{g} = \dbar\frac{1}{g} \text{,\quad }
    g \frac{1}{f}\dbar \frac{1}{g} = 0 \text{ and }
    \dbar\left(\frac{1}{f}\dbar\frac{1}{g}\right) = \dbar\frac{1}{f}\wedge \dbar\frac{1}{g}.
\end{equation}

\begin{ex} \label{expresidue}
    Let $Z \subseteq \Omega \subseteq \Cn$ be a reduced hypersurface defined by a
    holomorphic function $h$.
    On such a hypersurface, the \emph{Poincar\'e residue} $\omega$ of $dz/h$ is a
    meromorphic form, which can be defined by
    \begin{equation} \label{eqpresidue}
        i_* \omega = \dbar \frac{1}{h} \wedge dz.
    \end{equation}
    If we let $\tilde{\omega}$ be a meromorphic form on
    $\Omega$ such that $(dh/2\pi i) \wedge \tilde{\omega} = dz_1\wedge\dots\wedge dz_n =: dz$,
    then $\omega$ can alternatively be defined by $\omega := \tilde{\omega}|_Z$.
    This definition of $\omega$ does not depend on the choice of $\tilde{\omega}$. Considered as
    a meromorphic current, $\omega$ is $\dbar$-closed, see \cite{HP}.
    If $\partial h /\partial z_n$ does not vanish identically on any irreducible component of $Z$,
    then one can take $\omega = (-1)^{n-1}/(2\pi i\partial h/\partial z_n) dz_1\wedge \dots \wedge dz_{n-1}|_{Z}$.
    The Poincar\'e residue is a classical construction in mathematics, see for example \cite{Yg}.
    In this form it appears for example in \cite{HP}, and in similar forms in for
    example \cite{Barlet} and \cite{Herr}.
\end{ex}

\subsection{Almost semi-meromorphic and pseudomeromorphic currents} \label{ssectasmpm}

In $\C_z$ the principal value current $1/z^m$ can be defined as the analytic continuation
$|z^m|^{2\lambda}/z^m|_{\lambda=0}$, where by $|_{\lambda = 0}$ we mean that it is a
current-valued analytic function for $\Re \lambda \gg 0$, and $|_{\lambda = 0}$ denotes the
analytic continuation to $\lambda = 0$.
We can thus also define $\dbar(1/z^m)$ in the sense of currents,
which thus equals $\dbar |z^m|^{2\lambda}/z^m|_{\lambda=0}$. Hence, we can consider
tensor products of such one variable currents 
\begin{equation*}
    \tau = \dbar \frac{1}{z_1^{m_1}}\wedge\dots\wedge \dbar \frac{1}{z_k^{m_k}}
    \frac{\alpha}{z_{k+1}^{m_{k+1}}\dots z_N^{m_N}},
\end{equation*}
on $\C^N$, where $m_1,\dots,m_N$ are non-negative integers and $\alpha$ is a smooth
form with compact support. We call such a current an \emph{elementary current}.
Andersson and Wulcan introduced the following class of currents in \cite{AW2}.
\begin{df} \label{defpm}
    Let $Z$ be an analytic variety. A current $\mu$ on $Z$ is \emph{pseudomeromorphic},
    denoted $\mu \in \PM(Z)$ if it can be written as a locally finite sum of
    push-forwards $\pi_* \tau$ of elementary currents, where $\pi$ is a composition
    of modifications and open inclusions.
\end{df}
The definition in \cite{AW2} was for $Z$ a complex manifold, but allowing
$Z$ to be singular makes no difference. In \cite{AS2}, a slightly wider
definition was used, allowing more general push-forwards, but Definition~\ref{defpm} will
be sufficient for our purposes.

For pseudomeromorphic currents one can define natural restrictions to analytic subvarieties.
If $T \in \PM(Z)$, $V \subseteq Z$ is a subvariety of $Z$, and $h$ is a tuple of
holomorphic functions such that $V = Z(h)$, one defines
\begin{equation*}
    {\bf 1}_{Z\setminus V} T := |h|^{2\lambda} T |_{\lambda=0} \text{ and }
    {\bf 1}_V T := T - {\bf 1}_{Z\setminus V} T.
\end{equation*}
This definition is independent of the choice of tuple $h$, and ${\bf 1}_V T$
is a pseudomeromorphic current with support on $V$, see \cite{AW2}, Proposition~2.2.

A pseudomeromorphic current $\mu \in \PM(Z)$ is said to have the
\emph{standard extension property}, SEP, if ${\bf 1}_V \mu = 0$
for any subvariety $V \subseteq Z$ of positive codimension.
If $Z$ does not have pure dimension, we mean that $V$ has positive
codimension on each irreducible component of $Z$.

If $\alpha$ is a smooth form, and $T$ is a pseudomeromorphic current,
then ${\bf 1}_V(\alpha\wedge T) = \alpha \wedge{\bf 1}_V T$, and in particular,
if $T$ has the SEP, then $\alpha \wedge T$ also has the SEP.

An important property of pseudomeromorphic currents is that they satisfy
the following \emph{dimension principle}, Corollary~2.4 in \cite{AW2}.
\begin{prop} \label{proppmdim}
    If $T \in \PM(Z)$ is a $(p,q)$-current with support on a variety $V$,
    and $\codim V > q$, then $T = 0$.
\end{prop}

Given $f$ holomorphic on an analytic variety $Z$, as described in the introduction,
Herrera and Lieberman defined the principal value current $1/f$ on $Z$.
One way to define this is by
\begin{equation*}
    \frac{1}{f} . \phi := \int_{Z_{\rm reg}} \left.\frac{|f|^{2\lambda}}{f} \phi\right|_{\lambda=0},
\end{equation*}
where by $|_{\lambda = 0}$, we mean that right-hand side for $\Re \lambda \gg 0$ is analytic in
$\lambda$, and $|_{\lambda = 0}$ denotes the analytic continuation to $\lambda = 0$.
This way of defining the principal value current by analytic continuation goes
back to Atiyah, \cite{Atiyah}, and Bernstein-Gel'fand, \cite{BeGe}.
The proof of the existence of this analytic continuation relies on Hironaka's theorem
of resolution of singularities in order to write it as a locally finite sum of push-forwards
of elementary currents, and hence, principal value currents are pseudomeromorphic.

The product of a principal value current and a smooth form (i.e., the restriction
of a smooth form in the ambient space) is called a \emph{semi-meromorphic current}.
In \cite{AS2}, the authors introduce a generalization of this called \emph{almost
semi-meromorphic} currents.
\begin{df}
    A current $\mu$ on an analytic variety $Z$ is said to be \emph{almost semi-meromorphic}
    if $\mu = \pi_* \tilde{\mu}$, where $\tilde{\mu}$ is semi-meromorphic and
    $\pi : \tilde{Z} \to Z$ is a smooth modification of $Z$.
\end{df}
Since the class of pseudomeromorphic currents is closed under multiplication with smooth
functions and under push-forwards under modifications, almost semi-meromorphic currents
are pseudomeromorphic. By the dimension principle, principal value currents have the SEP,
and thus any semi-meromorphic current will also have the SEP.

\begin{df}
 The sheaf $\mathcal{W}_Z$ is the subsheaf of $\PM_Z$ of pseudomeromorphic currents
on $Z$ with the SEP on $Z$.    
\end{df}
In particular, almost semi-meromorphic currents are in $\mathcal{W}_Z$.
The fact that $\mathcal{W}_Z$ allows a natural multiplication with semi-meromorphic
currents will be crucial for the description of the currents we construct,
Proposition~2.7 in \cite{AS2}.

\begin{prop} \label{propprodasm}
    Let $\alpha$ be an almost semi-meromorphic current on $Z$. If $\mu \in \mathcal{W}(Z)$,
    then the current $\alpha\wedge \mu$, a priori defined where $\alpha$ is smooth,
    has a unique extension as a current in $\mathcal{W}(Z)$, which we also denote
    by $\alpha\wedge\mu$.
\end{prop}

\subsection{Andersson-Wulcan currents} \label{ssectaw}

Here we recall the construction in \cite{AW1}
of residue currents with prescribed annihilator ideals on complex manifolds.
Let $\mathcal{J} \subseteq \Ok$ be an ideal of holomorphic functions, and
let $(E,\varphi)$ be a Hermitian resolution of $\Ok/\mathcal{J}$, i.e.,
$(E,\varphi)$ is a free resolution
\begin{equation*}
    0 \xrightarrow[]{} E_N \xrightarrow[]{\varphi_N} E_{N-1} \xrightarrow[]{\varphi_{N-1}}
    \cdots \xrightarrow[]{\varphi_2} E_1 \xrightarrow[]{\varphi_1} E_0 \xrightarrow[]{} \Ok/\mathcal{J},
\end{equation*}
where the free modules $E_k \cong \Ok^{r_k}$ are equipped with Hermitian metrics.

To construct the current associated to $E$, one first defines, outside of
$Z = Z(\mathcal{J})$, right inverses $\sigma_k : E_{k-1} \to E_k$ to $\varphi_k$
which are minimal with respect to some metric on $E$, i.e.,
$\varphi_k \sigma_k|_{\Im \varphi_k} = \Id_{\Im \varphi_k}$, $\sigma_k = 0$ on
$(\Im \varphi_k)^\perp$, and $\Im \sigma_k \perp \ker \varphi_k$. One lets
$\sigma = \sigma_1 + \dots + \sigma_N$, and
\begin{equation} \label{equdef}
    u^E = \sigma + \sigma\dbar\sigma + \dots + \sigma(\dbar\sigma)^N.
\end{equation}
Letting $\nabla_{\End}$ be the morphism on $\D(\End E)$ induced by $\nabla = \varphi - \dbar$
by $\nabla_{\End}(\alpha) = \nabla\circ\alpha - \alpha\circ\nabla$,
one has that $\nabla_{\End} u^E = I_E$ outside of $Z$. The form $u^E$, which is smooth outside
of $Z$, has a current extension $U^E := |F|^{2\lambda} u^E |_{\lambda = 0}$
over $Z$, where $F \not\equiv 0$ is a holomorphic function vanishing at $Z$
and for $\Re \lambda \gg 0$, the right-hand side is is a (current-valued) analytic
function in $\lambda$, and $|_{\lambda = 0}$ denotes the analytic continuation to $\lambda = 0$.
The residue current $R^E$ associated to $E$ is defined as
\begin{equation} \label{eqrdef}
    R^E := I_E - \nabla_{\End} U^E.
\end{equation}
Alternatively, one could define $R^E$ by 
\begin{equation} \label{eqrlambda}
    R^E = \dbar |F|^{2\lambda} \wedge u^E |_{\lambda = 0}.
\end{equation}
See \cite{AW1} for more details.
From the proof of the existence of $U^E$ and $R^E$, it follows that they are
pseudomeromorphic.

Let $R^E_k$ denote the part of $R^E$ with values in $E_k$, i.e., $R^E_k$ is
a $E_k$-valued $(0,k)$-current. If $Z = Z(\mathcal{J})$, and $\codim Z = p$,
then we will in fact have that
\begin{equation} \label{eqrepN}
    R^E = R^E_p + \dots + R^E_N,
\end{equation}
where $N$ is the length of the free resolution $(E,\varphi)$.

The fundamental property of the current $R^E$ is the following, \cite{AW1}, Theorem~1.1.

\begin{thm} \label{thmreann}
    Let $R^E$ be the current associated to a free resolution $(E,\varphi)$
    of an ideal $\mathcal{J}$.
    Then $\ann R^E = \mathcal{J}$.
\end{thm}

In particular, if $\mathcal{J}$ is a complete intersection ideal,
$\mathcal{J} = \mathcal{J}(h_1,\dots,h_p)$, then the Koszul complex of $h$
is a free resolution of $\Ok/\mathcal{J}_Z$. In that case, both the Coleff-Herrera
product of $h$ and the current associated to the Koszul complex are currents with
annihilator equal to $\mathcal{J}$, and in fact they turn out to coincide.
Here, we identify the tuple $f$ with a section of $G^*$, where $G \cong \Ok^{\oplus p}$
with a frame $e_1,\dots,e_p$, so that $f = \sum f_i e_i^*$.

\begin{thm} \label{thmbmch}
    Let $f = (f_1,\dots,f_p)$ be a tuple of holomorphic functions defining a complete intersection.
    Let $R^f$ be the current associated to the Koszul complex of $f$,
    $R^f = \mu \wedge e_1\wedge \dots \wedge e_p$, and let $\mu^f$ be the Coleff-Herrera product
    of $f$. Then $\mu = \mu^f$.
\end{thm}

The current $R^f$ was originally defined by Passare, Tsikh and Yger in \cite{PTY}
(in a more direct way), referred to as a Bochner-Martinelli type residue current.
The equality of the Coleff-Herrera product and the Bochner-Martinelli type residue current
was originally proved in \cite{PTY}, Theorem~4.1, see also \cite{AndCH}, Corollary~3.2
for an alternative proof.

The definition of the Coleff-Herrera product and Bochner-Martinelli type current works
also in the singular case, and the equality of those in the case of complete intersection,
Theorem~\ref{thmbmch} also holds; the proof in \cite{AndCH} works also in the singular case,
see \cite{Lar1}, Theorem~6.4.

Note that from Theorem~\ref{thmreann} and Theorem~\ref{thmbmch}, the construction
by Andersson and Wulcan of a current with a prescribed annihilator ideal can be
seen as a generalization of the Coleff-Herrera product and the duality theorem
for Coleff-Herrera products.

We introduce the notation
\begin{equation} \label{eqRJX}
    R^\mathcal{J}_X := R^E\wedge \omega_X = \omega_X - \nabla(U^E \wedge \omega_X),
\end{equation}
where $R^E$ is the current associated to a minimal free resolution $(E,\varphi)$ of $\Ok/\mathcal{J}$,
and $\omega_X$ is a global holomorphic non-vanishing $(n,0)$-form on $X$ (for example if
$X$ is an open subset of $\Cn_z$, we can take $\omega_X = dz:= dz_1\wedge \dots dz_n$).
Note that since $\omega_X$ is assumed to be holomorphic and non-vanishing, we will
have that $\ann R^\mathcal{J}_X = \ann R^E = \mathcal{J}$, so in this setting,
the advantage of multiplying with the factor $\omega_X$ will not be very apparent,
but it will be important when we generalize this to singular varieties.

\subsection{A comparison formula for residue currents} \label{ssectcf}

An important tool in this article will be a comparison formula
for Andersson-Wulcan currents, \cite{Lar3}, which can be seen as a generalization
of the transformation law for Coleff-Herrera products.

Let $\mathcal{I} \subseteq \mathcal{J}$ be two ideals of holomorphic
functions, and let $(F,\psi)$ and $(E,\varphi)$ be free resolutions
of $\Ok/\mathcal{I}$ and $\Ok/\mathcal{J}$ respectively.
Since $\mathcal{I} \subseteq \mathcal{J}$, we have the natural surjection
$\pi : \Ok/\mathcal{I} \to \Ok/\mathcal{J}$.
By a rather straightforward diagram chase, one can show that there exists
a morphism of complexes $a : (F,\psi) \to (E,\varphi)$ making
the following diagram commute:
\begin{equation} \label{eqamorphism}
\xymatrix{
0 \ar[r] & E_N \ar[r]^{\varphi_N} & E_{N-1} \ar[r]^{\varphi_{N-1}}& \cdots \ar[r]^{\varphi_1} & E_0 \ar[r] &\Ok/\mathcal{J} \ar[r] & 0 \\
0 \ar[r] & F_N \ar[r]^{\psi_N} \ar[u]^{a_N}& F_{N-1}\ar[r]^{\psi_{N-1}} \ar[u]^{a_{N-1}} & \cdots \ar[r]^{\psi_1}& F_0 \ar[u]^{a_0}  \ar[r] & \Ok/\mathcal{I} \ar[u]^{\pi}\ar[r] & 0.
}
\end{equation}

The comparison formula, Theorem~1.2 in \cite{Lar3}, is expressed in terms of this morphism $a$.

\begin{thm} \label{thmRcomparisonIdeals}
    Let $\mathcal{I},\mathcal{J} \subseteq \Ok$ be two ideals of germs of holomorphic functions
    such that $\mathcal{I} \subseteq \mathcal{J}$, and let $(E,\varphi)$ and $(F,\psi)$ be
    minimal free resolutions of $\Ok/\mathcal{J}$ and $\Ok/\mathcal{I}$ respectively.
    Let $a : (F,\psi) \to (E,\varphi)$ be the morphism in \eqref{eqamorphism} induced by
    the natural surjection $\pi : \Ok/\mathcal{I} \to \Ok/\mathcal{J}$. Then,
    \begin{equation} \label{eqRcomparison}
        R^E a_0 = a R^F + \nabla_{\varphi} M,
    \end{equation}
    where $\nabla_{\varphi} = \sum \varphi_k - \dbar$,
    \begin{equation} \label{eqmdef}
        M = \left.\dbar |G|^{2\lambda} \wedge u^E \wedge a u^F \right|_{\lambda = 0},
    \end{equation}
    and $G$ is a tuple of holomorphic functions such that $\{ G = 0 \}$ contains
    the set where $(E,\varphi)$ or $(F,\psi)$ are not pointwise exact.
\end{thm}

\subsection{Singularity subvarieties of free resolutions} \label{ssectbef}

In the study of residue currents associated to free resolutions of ideals, an important
ingredient is certain singularity subvarieties associated to the ideal.
Given a free resolution $(E,\varphi)$ of an ideal $\mathcal{J}$, the variety $Z_k = Z^E_k$ is defined
as the set where $\varphi_k$ does not have optimal rank. These sets are independent of the choice of free
resolution. If $\codim Z(\mathcal{J}) = p$, then
$Z_k = Z$ for $k \leq p$, Corollary~20.12 in \cite{Eis}. In addition, Corollary~20.12 says that
$Z_{k+1} \subseteq Z_k$, and $\codim Z_k \geq k$ by Theorem~20.9 in \cite{Eis}.
In fact, Theorem~20.9 in \cite{Eis} is a characterization of exactness, the
Buchsbaum-Eisenbud criterion, which says that a generically exact complex of free
modules is exact if and only if $\codim Z_k \geq k$.

However, more precise information is obtained about which irreducible components $Z_k$
that are of maximal dimension. By Corollary 20.14, if $\codim V = k$, then $V \subseteq Z_k$
if and only if $\mathcal{I}_V \in \Ass \mathcal{J}$, i.e., if the ideal of holomorphic functions
vanishing on $V$ is an associated prime of $\Ok/\mathcal{J}$. In particular, if $\mathcal{J}$ is reduced,
$\Ass \mathcal{J}$ correspond exactly the irreducible components of $Z = Z(\mathcal{J})$. In that case, if
we let $W^d$ be the union of the irreducible components of $Z$ of codimension $p = n-d$, then $Z_p = W^d \cup Z_p'$,
where $\codim Z_p' \geq p+1$. If we consider $e > d$, then $\codim W^d \cap W^e \geq p+1$, so we get
that
\begin{equation} \label{eqzkprim}
    \codim W^e \cap Z_p \geq p+1.
\end{equation}

\subsection{Tensor products of free resolutions} \label{ssecttensprod}

In this section, we describe how under suitable conditions on ``proper'' intersection,
one can construct a free resolution of a sum of ideals from free resolutions of the
individual ideals. To begin with, let $(E,\varphi)$ and $(F,\psi)$ be two complexes.
The tensor product complex $(E\otimes F,\varphi\otimes\psi)$ is defined by
$(E\otimes F)_k = \oplus_{p+q=k} E_p\otimes F_q$ and
$(\varphi\otimes\psi)(\xi\otimes\eta) = \varphi(\xi)\otimes\eta + (-1)^i\xi\otimes\psi(\eta)$
if $\xi \in E_i$ and $\eta \in F_j$.
Note in particular that if $(E,\varphi)$ and $(F,\psi)$ are minimal free resolutions of ideals
$\mathcal{J}$ and $\mathcal{I}$, then $E_0 \cong \Ok \cong F_0$, and
$(\varphi\otimes\psi)_1 : E_1 \oplus F_1 \to \Ok$, $(\varphi\otimes\psi)_1 = \varphi_1 \oplus \psi_1$,
so if the tensor product complex is exact, it is a free resolution of $\mathcal{J}+\mathcal{I}$.
The tensor product complex will be exact if the corresponding singularity subvarieties
intersect properly in the following sense.

\begin{prop} \label{proptensorprodcomplexes}
    Let $(E,\varphi)$ and $(F,\psi)$ be free resolutions of ideal sheaves
    $\mathcal{J}$ and $\mathcal{I}$, and let $Z^E_k$ and $Z^F_l$ be the associated
    sets where $\varphi_k$ and $\psi_l$ do not have optimal rank.
    Then $(E \otimes F,\varphi\otimes \psi)$ is a free resolution of $\mathcal{I} + \mathcal{J}$
    if and only if $\codim ( Z^E_k \cap Z^F_l ) \geq k + l$ for all $k \geq \codim Z(\mathcal{J})$,
    $l \geq \codim Z(\mathcal{I})$.

    In addition, if $\mathcal{I}$ and $\mathcal{J}$ are Cohen-Macaulay ideals, and 
    $(E,\varphi)$ and $(F,\psi)$ are free resolutions of minimal length, then
    \begin{equation} \label{eqREFprod}
        R^{E\otimes F} = \left.(I_E - \nabla_\varphi(|G|^{2\lambda} u^E))\wedge R^F\right|_{\lambda=0},
    \end{equation}
    where $G$ is a tuple of holomorphic functions vanishing on $Z(\mathcal{J})$ but not identically
    on any irreducible component of $Z(\mathcal{I})$.
\end{prop}
A proof of the first part can be found in \cite{AndArk}, Remark~4.6, which we have reformulated slightly,
by only requiring the condition to hold for $k \geq \codim Z(\mathcal{J})$,
$l \geq \codim Z(\mathcal{I})$ instead of $k,l \geq 1$. However, this reformulation follows from the fact
that $Z^E_k = Z^E_p$ for $k \leq \codim Z(\mathcal{J})$ (and similarly for $Z^F_l$). The second part is part of Theorem~4.2 in \cite{AndArk}.

When $E$ and $F$ are equipped with Hermitian metrics, we will assume
that $E\otimes F$ is equipped with the product metric induced from
the metrics of $E$ and $F$.

\section{Currents with prescribed annihilator ideals on singular varieties of pure dimension} \label{sectpure}

Let $Z$ be an analytic subvariety of pure dimension $d$ of $\Omega\subseteq \Cn_z$.
We first consider the current $R^{\mathcal{I}_Z}$ associated to $\mathcal{I}_Z$,
the ideal of holomorphic functions on $\Omega$ vanishing on $Z$.
In \cite{AS2}, Andersson and Samuelsson showed that there exists what they
call a \emph{structure form} $\omega_Z$ associated to $Z$, generalizing the
Poincar\'e residue in Section~\ref{ssectcurrvar}.
The following part of Proposition~3.3 in \cite{AS2} will be sufficient for our purposes.

\begin{prop} \label{propomegapure}
    Let $(F,\psi)$ be a Hermitian resolution of $\Ok_\Omega/\mathcal{I}_Z$,
    and let $R^{\mathcal{I}_Z}$ be the associated residue current.
    Then there exists an almost semi-meromorphic current
    \begin{equation*}
        \omega_Z = \omega_0 + \dots + \omega_{d-1}
    \end{equation*}
    on $Z$, where $\dim Z = d$, $\codim Z = p$, and $\omega_r$ has bidegree
    $(d,r)$ and takes values in $F_{p+r}$, such that
    \begin{equation} \label{eqomegaR}
        i_* \omega_Z = R^{\mathcal{I}_Z} \wedge dz,
    \end{equation}
    where $i : Z \to \Omega$ is the inclusion and $dz := dz_1\wedge \dots \wedge dz_n$.
\end{prop}

The structure form $\omega_Z$ plays an important role in \cite{AS2} and \cite{AS3} related
to the $\dbar$-equation on singular varieties. It also appears (more implicitly) in
\cite{ASS}, related to the Brian\c{c}on-Skoda theorem on a singular variety.

Let $\mathcal{J}\subseteq\Ok_Z$ be an ideal. We will use the comparison formula
from Section~\ref{ssectcf} in order to construct intrinsically on $Z$ the current
with the prescribed annihilator ideal in terms of almost semi-meromorphic currents.
Let $\mathcal{\tilde{J}} \subseteq \Ok_\Omega$ be the maximal lifting of the ideal $\mathcal{J}$,
i.e., the largest ideal $\mathcal{\tilde{J}}$ such that $i^* \mathcal{\tilde{J}} = \mathcal{J}$,
where $i^* : \Ok_\Omega \to \Ok_Z$ is induced by the inclusion $i : Z \to \Omega$.
Note that $\mathcal{I}_Z \subseteq \mathcal{\tilde{J}}$ (since $i^* \mathcal{I}_Z = 0$),
so if $(E,\varphi)$ and $(F,\psi)$ are free resolutions of $\mathcal{\tilde{J}}$ and $\mathcal{I}_Z$
respectively, we get a morphism of complexes
\begin{equation} \label{eqadef}
    a : (F,\psi) \to (E,\varphi)
\end{equation}
extending the natural surjection $\pi_* : \Ok_\Omega/\mathcal{I}_Z \to \Ok_\Omega/\mathcal{\tilde{J}}$
as in \eqref{eqamorphism}.

On $Z \setminus Z^E_{p+1}$, let
\begin{equation} \label{eqnudef}
    \nu := \sum_{m \geq k \geq p+1} \sigma_m^E\dbar\sigma^E_{m-1}\dots\dbar\sigma^E_k.
\end{equation}
Note that since $\sigma^E_l$ and $\dbar\sigma^E_l$ are smooth outside $Z_l^E$, we get
that $\nu$ is smooth outside $Z_{p+1}^E$. Note that $Z^E_{p+1} \subset Z_{\rm sing}$
because on the regular part, the Koszul complex of coordinate functions defining the variety is
a free resolution of length $p$. Thus, $\nu$ is defined and smooth on $Z_{\rm reg}$.

Now we are ready to state our main theorem.

\begin{thm} \label{thmmain1}
    Let $Z \subseteq \Omega \subseteq \C^n$ be an analytic subvariety of $\Omega$ of pure dimension,
    where $\Omega$ is an open set in $\C^n$. Let $\mathcal{J} \subseteq \Ok_Z$
    be an ideal.
    Then $\nu$ defined by \eqref{eqnudef} has an extension as an almost
    semi-meromorphic current to $Z$, which we denote by $V^E$.
    If we let 
    \begin{equation} \label{eqrpuredef}
        R^\mathcal{J}_Z  := a \omega_Z - \nabla( V^E \wedge a \omega_Z ),
    \end{equation}
    where $a : (F,\psi) \to (E,\varphi)$ is the morphism in \eqref{eqadef},
    then 
    \begin{equation} \label{eqannRJ}
        \ann_{\Ok_Z} R^\mathcal{J}_Z  = \mathcal{J}.
    \end{equation}
    Moreover,
    \begin{equation} \label{eqRJpushforward}
        i_* R^\mathcal{J}_Z = R^{\mathcal{\tilde{J}}}_\Omega,
    \end{equation}
    where $\mathcal{\tilde{J}} \subseteq \Ok_\Omega$ is the maximal lifting of
    $\mathcal{J}$, and $R^{\mathcal{\tilde{J}}}_\Omega$ is the current associated
    to $\mathcal{\tilde{J}}$ as in \eqref{eqRJX}.
\end{thm}

\begin{proof}
    By applying the comparison formula \eqref{eqRcomparison} to $a : (F,\psi) \to (E,\varphi)$,
    and taking the wedge product with $\omega_\Omega = dz$, we get that
    \begin{equation} \label{eqcompformulaz}
        R^{\mathcal{\tilde{J}}}_\Omega = a R^{\mathcal{I}_Z}_\Omega + \nabla(M\wedge \omega_\Omega).
    \end{equation}
    If we show that $M\wedge\omega_\Omega$ in \eqref{eqcompformulaz} is the push-forward
    of $-V^E\wedge a\omega_Z$, then \eqref{eqRJpushforward} will follow from 
    \eqref{eqcompformulaz} together with Proposition~\ref{propomegapure}, and
    \eqref{eqannRJ} follows from the fact that
    $\ann_\Ok R^{\mathcal{\tilde{J}}}_\Omega = \mathcal{\tilde{J}}$.

    The proof that $M\wedge\omega_\Omega$ is the push-forward of $-V^E\wedge a\omega_Z$ will
    be rather similar to the proof of Lemma~6.2 in \cite{Lar3} (which says that $M\wedge\omega_\Omega$
    corresponds to a current on $Z$). We let
    \begin{equation*}
        M^l_k = \dbar |G|^{2\lambda} \wedge \dbar\sigma_k^E\dbar\sigma_{k-1}^E\dots\sigma_{l+1}^E a_l \sigma_l^F\dbar\sigma_{l-1}^F\dots\dbar\sigma_1^F|_{\lambda=0}.
    \end{equation*}
    Note that by using that $\dbar \sigma_{j+1} \sigma_j = \sigma_{j+1}\dbar\sigma_{j}$, it follows that
    the current $M$ in \eqref{eqmdef} is exactly $\sum_{l<k} M^l_k$.
    However, in the definition of $M^l_k$ we also allow $k = l$, which we interpret
    as containing no $\sigma^E$'s at all. The reason we allow $k=l$ is that we use it
    as a starting point for an inductive argument. 

    If $j \geq p+1$, then $\sigma^E_j$ and $\dbar\sigma^E_j$ are smooth outside
    $Z_j \subseteq Z_{p+1}$, which has codimension $\geq p+1$, and since
    $\codim Z = p$, $Z_{p+1}$ has codimension $\geq 1$ in $Z$.
    As in the proof of Proposition~3.3 in \cite{AS2}, one sees that the restrictions
    of $\sigma^E_j$ and $\dbar\sigma^E_j$ to $Z$ are almost semi-meromorphic on $Z$.
    Hence, when can define
    \begin{equation} \label{eqvlkdef}
        V^l_k := \left\{
        \begin{array}{cc} i^*\dbar\sigma^E_k\dots i^*\dbar\sigma^E_{l+2}i^*\sigma^E_{l+1} & l \geq p \\
                          0 & l < p
        \end{array} \right.
    \end{equation}
    as a product of almost semi-meromorphic currents on $Z$ by Proposition~\ref{propprodasm}.

    Note that if $l \geq p$, we have outside of $Z_{p+1}$ that
    \begin{equation} \label{eqmlkomega}
        M^l_k\wedge \omega_\Omega = -\dbar\sigma^E_k\dots\dbar\sigma^E_{l+2}\sigma^E_{l+1} a_l R^F_l\wedge \omega_\Omega
        = i_*(V^l_k \wedge a_l \omega_{l-p}),
    \end{equation}
    where the minus sign in the first equality is due to
    $\dbar\sigma^E_k\dots\dbar\sigma^E_{l+2}\sigma^E_{l+1}$ being of odd degree and
    hence anti-commuting with $\dbar |G|^{2\lambda}$, and the second equality is
    due to \eqref{eqomegaR} and \eqref{eqvlkdef}.
    
    The right-hand side of \eqref{eqmlkomega} has a unique extension as a
    product of almost semi-meromorphic currents by Proposition~\ref{propprodasm}
    and this extension has the SEP with respect to $Z$. Hence, this extension will
    coincide with $M^l_k\wedge\omega_\Omega$ if we can prove that $M^l_k\wedge\omega_\Omega$
    also has the SEP with respect
    to $Z$. When $l < p$, $V^l_k = 0$, so we thus also want to prove that $M^l_k = 0$ if $l<p$.
    We will prove both these statements, i.e., that $M^l_k = 0$ if $l < p$, and that
    $M^l_k\wedge\omega_\Omega$ has the SEP with respect to $Z$ if $l \geq p$, simultaneously by induction over $k-l$.

    For $k=l$, $M^l_l$ is a pseudomeromorphic $(0,l)$-current (note that $M^l_k$ is a $(0,k-1)$-current
    when $k > l$, but an $(0,k)$-current when $k=l$) with support on
    $Z$, which has codimension $p$, so if $l <p$, then $M^l_l = 0$ by the dimension
    principle. For $l \geq p$, note that $M^l_l = a_l R^{\mathcal{I}_Z}_l$, so
    $M^l_l\wedge\omega_\Omega = i_* (a_l \omega_{l-p})$, and since $\omega_{l-p}$ is almost semi-meromorphic
    on $Z$, it has the SEP with respect to $Z$.

    We thus now assume that $M_{k-1}^l = 0$ for $l<p$, and $M_{k-1}^l\wedge\omega_\Omega$ has the SEP with respect to $Z$
    for $l\geq p$, and we want to prove the same for $M_k^l$.
    We first consider the case $k = l+1$. Then $M_{l+1}^l = \sigma_{l+1}^E M^l_l$ outside
    of $Z_{l+1}$. If $k \leq p$, we thus get that $\supp M_{l+1}^l \subseteq Z_{l+1}$,
    and since $M_{l+1}^l$ is a pseudomeromorphic $(0,l)$-current, we get by the dimension
    principle that $M_{l+1}^l = 0$. If $k \geq p+1$, then since $M_{l+1}^l = \sigma_{l+1}^E M^l_l$
    outside of $Z_{l+1}$, and $M^l_l\wedge\omega_\Omega$ has the SEP with respect to $Z$, we get that
    $\supp {\bf 1}_V M_{l+1}^l\wedge\omega_\Omega \subseteq Z_{l+1}$ if $V$ is a subvariety of $Z$
    of codimension $\geq 1$.
    Since ${\bf 1}_V M_{l+1}^l\wedge\omega_\Omega$ is a pseudomeromorphic $(n,l)$-current, it is $0$
    by the dimension principle, i.e., $M_{l+1}^l\wedge\omega_\Omega$ has the SEP with respect to $Z$.
    The argument for $k > l+1$ follows in essentially
    the same way as for $k=l+1$, with the change that $M^l_k = \dbar \sigma^ E_k M^l_{k-1}$ instead.
    The case when $k > p$, while $l < p$ only appears when $k \geq \ell + 2$, and is handled in the
    following way. In that case, by induction, $M^l_{k-1} = 0$. Thus, since
    $M^l_k = \dbar \sigma^E_k M^l_{k-1}$ outside of $Z_k$, $\supp M^l_k \subseteq Z_k$, and since
    $M^l_k$ is a pseudomeromorphic $(0,k-1)$-current, we get by the dimension principle that $M^l_k = 0$.

    Thus, we see from \eqref{eqmlkomega} that $i_* (-V^E\wedge a \omega_Z) = M\wedge\omega_\Omega$
    since $M = \sum_{k > l} M^l_k$, $V^E = \sum_{k > l \geq p} V^l_k$
    and $M^l_k = 0$ if $l<p$.
\end{proof}

We now consider some examples of this construction.

\begin{ex} \label{expurecm}
    Let $Z \subseteq \Omega$ be a Cohen-Macaulay variety, i.e., if $\codim Z = p$,
    then $\Ok_\Omega/\mathcal{I}_Z$ has a free resolution of length $p$.
    Let $\mathcal{J} \subseteq \Ok_Z$ be an ideal with a lifting $\mathcal{\widehat{J}}$
    of $\mathcal{J}$ to $\Ok_\Omega$ such that if $\codim Z(\mathcal{J}) = q$ in $Z$,
    then $\mathcal{\widehat{J}}$ is a Cohen-Macaulay ideal of codimension $q$ in $\Omega$.
    Note that we want to take the lifting $\mathcal{\widehat{J}}$ to be as small as possible,
    in contrast to Theorem~\ref{thmmain1}, where we take the largest lifting.

    One example is when $\mathcal{J} = \mathcal{J}(f_1,\dots,f_q)\subseteq \Ok_Z$
    is a complete intersection ideal.

    With these conditions, we can apply Proposition~\ref{proptensorprodcomplexes} to the ideals
    $\mathcal{I}_Z$ and $\mathcal{\widehat{J}}$, so the tensor product complex $(E\otimes F,\varphi\otimes\psi)$
    is a free resolution of $\Ok_\Omega/(\mathcal{I}_Z + \mathcal{\widehat{J}}) = \Ok_\Omega/\mathcal{\tilde{J}}$,
    where $(E,\varphi)$ and $(F,\psi)$ are free resolutions of $\Ok_\Omega/\mathcal{\widehat{J}}$ and
    $\Ok_\Omega/\mathcal{I}_Z$ respectively.

    Since $i_* \omega_Z = R^F\wedge dz$ and $i_* R^\mathcal{J}_Z = R^{E\otimes F}\wedge dz$, we thus get by
    \eqref{eqREFprod} that
    \begin{equation*}
        i_* R^\mathcal{J}_Z = \left.i_*( (I_E -\nabla_\varphi(|G|^{2\lambda} u^E))\wedge \omega_Z)\right|_{\lambda=0}.
    \end{equation*}
    Since $Z$ is Cohen-Macaulay, $\omega_Z=\omega_0$, so $\nabla_\varphi(\omega_Z) = -\dbar\omega_Z = 0$
    (note the $\varphi$, not $\psi$), since $\dbar \omega_0 = \psi_{p+1}\omega_1 = 0$. In addition,
    $i_*$ is injective on currents on $Z$, so
    \begin{equation} \label{eqRJCM}
        R^\mathcal{J}_Z = \omega_Z -\nabla_\varphi(\left.|G|^{2\lambda} u^E \wedge \omega_Z\right|_{\lambda=0}).
    \end{equation}

    From \eqref{eqRJCM}, we can see that the current $R^f_Z$ we defined in Proposition~\ref{propprincipalonhyper}
    in the introduction is the current given by Theorem~\ref{thmmain1}. When $Z$ is a reduced hypersurface
    defined by $h$, then $R^Z = \dbar(1/h)$, so the structure form $\omega_Z$ becomes just the Poincar\'e residue
    of $dz/h$ on $Z$. In addition, the free resolution $(E,\varphi)$ of $\Ok/\mathcal{J}(f)$ becomes just
    the complex $\Ok \stackrel{(f)}{\to} \Ok$. Hence,
    \begin{equation*}
        R^{\mathcal{J}(f)}_Z = \omega_Z - (f-\dbar)\left(\frac{1}{f}\omega_Z\right)
        = \dbar\left(\frac{1}{f}\omega_Z\right).
    \end{equation*}
\end{ex}

The structure form $\omega_Z$ here plays a bit similar role as in \cite{AS2}.
In \cite{AS2}, for example $\dbar$-closedness for a current $T \in \mathcal{W}_Z$
is expressed as $\dbar (T\wedge \omega_Z) = 0$, not just $\dbar T = 0$.
In the case of $(0,0)$-currents, $\dbar$-closed currents in this sense
become just holomorphic functions, i.e., as expected from the smooth case,
while there can exist $\dbar$-closed $(0,0)$-currents in the usual sense which are not
holomorphic functions when $Z$ is singular. Here, we get that the annihilator of $\dbar(1/f)$ might be
larger than the ideal generated by $f$, while adding $\omega_Z$, the annihilator of
$\dbar( (1/f)\omega_Z)$ becomes exactly $f$.

We finish this section with an example not covered by Example~\ref{expurecm}.

\begin{ex}
    Consider the cusp $Z = \{ z^3 - w^2 = 0 \} \subseteq \C^2$, and the maximal ideal at $0$,
    $\mathfrak{m} = \mathcal{J}(z,w) \subseteq \Ok_Z$. Note that since $z^3 - w^2 \in \mathcal{J}(z,w)$,
    the maximal lifting of $\mathfrak{m}$ to $\Ok = \Ok_{\C^2}$ equals
    $\mathfrak{\tilde{m}} = \mathcal{J}(z,w) \subseteq \Ok$.
    It is easily verified that the morphism $a : (F,\psi) \to (E,\varphi)$ from \eqref{eqadef}, where
    $(F,\psi)$ and $(E,\varphi)$ are free resolutions of $\Ok/\mathcal{I}_Z$ and
    $\Ok/\mathfrak{\tilde{m}}$, becomes
    \begin{equation} \label{eqcuspmaxideal}
    \xymatrix{
        0 \ar[r] & \Ok \ar[r]^{\varphi_2} & \Ok^{\oplus 2} \ar[r]^{\varphi_1} & \Ok \ar[r]
        &\Ok/\mathfrak{\tilde{m}} \ar[r] & 0 \\
        & 0 \ar[r] & \ar[u]^{a_1} \Ok \ar[r]^{\psi_1}& \Ok \ar[u]^{a_0}  \ar[r] &
        \Ok/\mathcal{I}_Z \ar[u]^{\pi}\ar[r] & 0,
    }
    \end{equation}
    where
    \begin{equation*}
        \varphi_2 = \left( \begin{array}{c} -w \\ z \end{array} \right) \text{, }
            \varphi_1 = \left( \begin{array}{cc} z & w \end{array} \right) \text{, }
                \psi_1 = \left( \begin{array}{c} z^3-w^2 \end{array} \right) \text{ and }
                a_1 = \left( \begin{array}{c} z^2 \\ -w \end{array} \right).
    \end{equation*}
    Choosing the trivial metric on $E$, the minimal right-inverse $\sigma_2$ of $\varphi_2$ is
    $( \begin{array}{cc} -\bar{w} & \bar{z} \end{array} )/\left(|z|^2 + |w|^2\right)$.
    Since $Z$ is a reduced hypersurface defined by $z^3-w^2$, the structure form $\omega$
    becomes $2\pi i dz/(2w)|_{Z}$ as in Example~\ref{expresidue}. 

    We let $\tau : \C \to Z$, $\tau(t) = (t^2,t^3)$, which is a smooth modification of $Z$
    (in fact, it is the normalization of $Z$). Then, one can verify that $\tau^*(\sigma_2a_1) = -t$,
    and since $\tau^*(dz/(2w)) = dt^2/(2t^3) = dt/t^2$, we get that $\tau^*(V^E a \omega) = -2\pi idt/t$.
    Thus, 
    \begin{equation*}
        V^E a \omega = -2\pi i\tau_* (dt/t)
    \end{equation*}
    (since $\tau_* \tau^* = \Id$ for currents with the SEP on $Z$, where $\tau$ is a modification).
    Since $\supp R^\mathfrak{m}_Z \subseteq Z(\mathfrak{m}) = \{ 0 \}$, we get by the
    dimension principle that $R^\mathfrak{m}_Z = -\dbar(V^Ea_1\omega)$, since the right-hand
    side here is the only part of $R^\mathfrak{m}_Z$ as defined by \eqref{eqrpuredef} of
    bidegree $(*,1)$ on $Z$. Thus,
    \begin{equation*}
        R^{\mathfrak{m}}_Z = 2\pi i\dbar\tau_*\left(\frac{dt}{t}\right) =
        2\pi i\tau_*\left(\dbar\left(\frac{dt}{t}\right)\right) = \tau_* ((2\pi i)^2[0]) = (2\pi i)^2[0].
    \end{equation*}
    This could also have been seen directly in this case from \eqref{eqRJpushforward}, since
    \begin{equation*}
      R^{\mathfrak{\tilde{m}}}_{\C^2} = \dbar(1/w)\wedge\dbar(1/z)\wedge dz \wedge dw = (2\pi i)^2[0].
    \end{equation*}
    Note that since $\tau^*(dz/(2z)) = dt/t$, we can also express this as
    \begin{equation*}
        R^{\mathfrak{m}}_Z = \dbar\left( 2\pi i\left.\frac{dz}{2z}\right|_{Z}\right).
    \end{equation*}
\end{ex}

\section{The construction in the case that $Z$ is smooth} \label{sectsmooth}

Note the similarity of the definition of $R^\mathcal{J}_Z$ in \eqref{eqRJX} and \eqref{eqrpuredef}.
In fact, it is easy to see that if $Z = \Omega \subseteq \Cn$, then the definitions of $R$ from
\eqref{eqrdef} and \eqref{eqrpuredef} coincide, since then, $(F,\psi)$ becomes just
$F_0 \cong \Ok$, and $a = a_0$ is the isomorphism $a_0 : F_0 \cong \Ok \cong E_0$,
$V^E = U^E$ and $\omega_Z = dz$.
In fact, even more holds.
\begin{prop} \label{propsubmanifold}
    Let $Z$ be a smooth subvariety of $\Omega$.
    Let $\mathcal{J}$ be an ideal in $\Ok_Z$.
    Then $R^\mathcal{J}_Z$ for an ideal $\mathcal{J} \subseteq \Ok_Z$ defined intrinsically
    on $Z$ as in \eqref{eqRJX} as the current associated to a free resolution on $Z$
    coincides with the current defined in \eqref{eqrpuredef}.
\end{prop}

In particular, it is motivated to use the same notation $R^\mathcal{J}_Z$ for both
the currents defined by \eqref{eqRJX} and \eqref{eqrpuredef}.

A key part in proving Proposition~\ref{propsubmanifold} is the following lemma
about currents associated to specific product complexes.

\begin{lma} \label{lmaproduct}
    Let $\mathcal{J} \subset \Ok_Z$ be an ideal, where $Z$ is the smooth subvariety
    of $\Omega \subset \Cn_z \times \C^m_w$ defined by $Z = \{ w_1 = \dots = w_m = 0 \}$.
    Let $(E,\varphi)$ be a free resolution of $\Ok_Z/\mathcal{J}$, with $E_0 \cong \Ok_Z$,
    and let $(\widehat{E},\widehat{\varphi}) = (\pi^* E,\pi^* \varphi)$,
    where $\pi : \Cn_z \times \C^m_w \to \Cn_z$ is the projection.
    Let $(F,\psi)$ be the Koszul complex of $(w_1,\dots,w_m)$.
    Then,
    \begin{equation*}
        R^{\widehat{E}\otimes F} = R^E \otimes R^F.
    \end{equation*}
\end{lma}

We begin by proving Proposition~\ref{propsubmanifold} using this lemma,
and then come back to the proof of Lemma~\ref{lmaproduct}.

\begin{proof}[Proof of Proposition~\ref{propsubmanifold}]
    We assume that locally, $Z = \{ w_1 = \dots = w_m = 0 \} \subseteq \C^n_z \times \C^m_w$,
    i.e., $z$ are local coordinates on $Z$ and $\mathcal{I}_Z = \mathcal{J}(w_1,\dots,w_m)$.
    We let $R$ be the current $R^\mathcal{J}_Z$ defined by \eqref{eqrpuredef},
    and let $R'$ be the current $R^\mathcal{J}_Z$ defined by \eqref{eqRJX}.
    We let $\mathcal{\widehat{J}} = \pi^* \mathcal{J} = \Ok_{\C_z^n\times \C^m_w} \otimes_{\Ok_{\C^n_z}} \mathcal{J}$
    be the ideal $\mathcal{J}$ considered as an $\Ok = \Ok_{\C^n_z\times\C^m_w}$-module.
    We also let $(\widehat{E},\widehat{\varphi}) := (\pi^*E,\pi^*\varphi)$,
    where $(E,\varphi)$ is a free resolution of $\Ok_{\C^n_z}/\mathcal{J}$ and $\pi : \C^n_z\times\C^m_w \to \C^n_z$
    is the projection. Then $(\widehat{E},\widehat{\varphi})$
    is a free resolution of $\Ok_{\C^n_z\times\C^m_w}/\mathcal{\widehat{J}}$ since $\Ok_{\C^n_z\times\C^m_w}$
    is a flat $\Ok_{\C^n_z}$-module, see \cite{Fis}, Proposition~3.17.

    Let $(F,\psi)$ be the Koszul complex of $(w_1,\dots,w_m)$, which is a free resolution of $\Ok/\mathcal{I}_Z$.
    The maximal lifting $\mathcal{\tilde{J}}$ of $\mathcal{J}$ equals $\mathcal{\widehat{J}} + \mathcal{I}_Z$,
    so by Proposition~\ref{proptensorprodcomplexes}, $(\widehat{E}\otimes F,\widehat{\varphi}\otimes \psi)$
    is a free resolution of $\Ok/\mathcal{\tilde{J}}$.
    Thus, by \eqref{eqRJpushforward},
    \begin{equation*}
        i_* R = R^{\widehat{E}\oplus F}\wedge dz\wedge dw,
    \end{equation*}
    and by Lemma~\ref{lmaproduct},
    \begin{equation*}
        i_* R = R^{E} \otimes R^F \wedge dz\wedge dw.
    \end{equation*}
    By Theorem~\ref{thmbmch}, $R^F = \mu^w \wedge e_1\wedge\dots\wedge e_m$,
    and by the Poincar\'e-Lelong formula, $\mu^w \wedge dw = (2\pi i)^m [ w = 0 ]$,
    so
    \begin{equation*}
        i_* R = c R^E \otimes [w = 0] \wedge e_1\wedge\dots\wedge e_m,
    \end{equation*}
    for some non-zero constant $c$.
    Note also that $i_* R' = R^E \wedge [w=0]$, so $i_* R' = i_* R$,
    (up to $c e_1\wedge \dots \wedge e_m$), i.e., $R' = R$
    (up to isomorphism).
\end{proof}

The proof of Lemma~\ref{lmaproduct} follows easily from the following lemma.
\begin{lma} \label{lmaproduct2}
    Using the notation of Lemma~\ref{lmaproduct}, and assuming that $m=1$,
    and $M$ is the current in \eqref{eqmdef} obtained by applying the comparison
    theorem to the morphism $a : (F,\psi) \to (\widehat{E}\otimes F,\widehat{\varphi}\otimes \psi)$, then
    \begin{equation*}
        M = -U^E \wedge a_1 R^F.
    \end{equation*}
\end{lma}

\begin{proof}[Proof of Lemma~\ref{lmaproduct}]
    By induction, it is easily seen that it suffices to prove Lemma~\ref{lmaproduct}
    in the case when $m=1$.
    Let $a : (F,\psi) \to (\widehat{E}\otimes F,\widehat{\varphi}\otimes \psi)$
    be the morphism of complexes given by $a_k : F_k \to E_0 \otimes F_k$, $a_k(\xi) = 1\otimes \xi$.
    By the comparison formula,
    \begin{equation*}
        R^{\widehat{E}\otimes F} = a_1 R^F - \nabla M,
    \end{equation*}
    and since $\nabla(U^E) = I - R^E$ and $\nabla R^F = 0$,
    we get by Lemma~\ref{lmaproduct2} that
    \begin{equation*}
         R^{\widehat{E}\otimes F}= a_1 R^F - (a_1 R^F - R^E \wedge a_1 R^F) = R^E \wedge a_1 R^F.
    \end{equation*}
    Since $a_1 \xi = 1\otimes \xi$, we get that
    \begin{equation*}
        R^{\widehat{E}\otimes F} = R^E \otimes R^F.
    \end{equation*}
\end{proof}

In order to prove Lemma~\ref{lmaproduct2}, we first need
to elaborate a bit on how the $\sigma$'s are defined, and
in particular, how the $\sigma^{\widehat{E}\otimes F}$'s
are related to the $\sigma^{\widehat{E}}$'s.

In general, for a generically exact complex $(G,\eta)$
of Hermitian vector bundles, $\sigma^G_k$, which is smooth outside of $Z^G_k$, can be defined
as
\begin{equation} \label{eqsigmaexpr}
    \sigma^G_k = (\eta_{k+1}\eta_{k+1}^* + \eta_k^*\eta_k)^{-1} \eta_k^*,
\end{equation}
where $\eta_k^*$ and $\eta_{k+1}^*$ are the adjoint morphisms
of $\eta_k$ and $\eta_{k+1}$ induced by the Hermitian metrics.
This can be seen by the fact that outside of $Z^G_k$, $\im \eta_k^*$ is the orthogonal
complement of $\im \eta_{k+1}$, and the restrictions of $\eta_{k+1}\eta_{k+1}^*$ and $\eta_k^*\eta_k$
to $\im \eta_{k+1}$ and $\im \eta_k^*$ are invertible, and one then easily verifies that
$\sigma^G_k$ defined by \eqref{eqsigmaexpr} is a minimal right inverse of $\eta_k$.

We begin by expressing the $\sigma^G$'s more explicitly when
$(G,\eta) = (\widehat{E} \otimes F, \widehat{\varphi}\otimes \psi)$,
and $(F,\psi)$ is the Koszul complex of $w$.

Then,
\begin{equation*}
    G_k = E_k \oplus E_{k-1} \text { for } k \geq 1 \text { and }
    G_0 = E_0.
\end{equation*}
(identifying $F_0 \cong F_1 \cong \Ok$), and
\begin{equation} \label{eqeta1}
    \eta_k = \left(
    \begin{array}{cc} \varphi_k & (-1)^{k-1} w \Id_{E_{k-1}} \\
        0 & \varphi_{k-1} \end{array} \right)
          \text{ for } k \geq 2 \text{ and }
        \eta_1 = \left( \begin{array}{cc} \varphi_1 & w \Id_{E_0}
        \end{array} \right).
    \end{equation}
One then calculates that
\begin{equation} \label{eqeta2}
    \eta_k^* \eta_k + \eta_{k+1}\eta_{k+1}^* =
    \left( \begin{array}{cc} \varphi_k^* \varphi_k + \varphi_{k+1}\varphi_{k+1}^* + |w|^2 & 0 \\
        0 & \varphi_k\varphi_k^* + \varphi_{k-1}^* \varphi_{k-1} + |w|^2 \end{array} \right).
\end{equation}
We can then use the following lemma to relate the $\sigma^G$'s with the $\sigma^E$'s.
\begin{lma}
    For $(G,\eta)$ a generically exact complex with $Z^G_{k+1} \subset Z^G_{k}$,
\begin{equation} \label{eqsigmawexpr}
    \sigma^G_{k,w} := (|w|^2 + \eta_{k+1}\eta_{k+1}^* + \eta_k^*\eta_k)^{-1} \eta_k^*
\end{equation}
depends smoothly on $(z,w)$ for $(z,w)$ outside of $Z^G_k\times \{ 0 \}$,
and
\begin{equation} \label{eqsigmawexpr2}
    \sigma^G_{k,w} = \sigma^G_k + |w|^2\alpha,
\end{equation}
for $z$ outside of $Z^G_k$, where $\alpha$ depends smoothly on $z$ and $w$.
\end{lma}

\begin{proof}
The first statement follows from the fact that 
$|w|^2 + \eta_{k+1}\eta_{k+1}^* + \eta_k^*\eta_k$ is strictly
positive definite if $z \notin Z^G_k$ or $w \neq 0$,
and the inverse thus depends smoothly on $(z,w)$.

It thus remains to prove \eqref{eqsigmawexpr2}.
To begin with, if $A$ is an invertible matrix such that
$A$ and $A + \epsilon I$ are invertible, then
\begin{equation*}
    (A + \epsilon I)^{-1} = A^{-1} - \epsilon A^{-1}(A+\epsilon I)^{-1}.
\end{equation*}
In particular, if we let $A = \eta_{k+1}\eta_{k+1}^* + \eta_k^*\eta_k$
and $\epsilon = |w|^2$, then for $z \notin Z^G_k$, $A$ is positive definite,
so $A$ and $A+\epsilon I$ are invertible, and the inverses
depend smoothly on $(z,w)$, and thus \eqref{eqsigmawexpr2} follows by \eqref{eqsigmaexpr},
where
\begin{equation*}
    \alpha = -(\eta_{k+1}\eta_{k+1}^* + \eta_k^*\eta_k)^{-1}(\eta_{k+1}\eta_{k+1}^* + \eta_k^*\eta_k + |w|^2)^{-1}.
\end{equation*}
\end{proof}

By \eqref{eqeta1} and \eqref{eqeta2},
\begin{align} \label{eqsigmagexpr2}
    \sigma^{\widehat{E}\otimes F}_k &= \left( \begin{array}{cc} * & 0 \\ * & (\varphi_k\varphi_k^* + \varphi_{k-1}^* \varphi_{k-1} + |w|^2)^{-1} \end{array} \right)
        \left( \begin{array}{cc} * & 0 \\ * & \varphi_{k-1}^* \end{array} \right) \\
\nonumber            &= \left( \begin{array}{cc} * & 0 \\ * & \sigma^E_{k-1,w} \end{array} \right),
\end{align}
where we by $*$ mean that we ignore the entries in those positions.

\begin{proof}[Proof of Lemma~\ref{lmaproduct2}]
    Note that since $R^F = \dbar(1/w)\wedge e$, if write
    an element $\xi_1 \oplus \xi_2 \in G_k = E_k \oplus E_{k-1}$
    in vector form
    \begin{equation*}
        \xi_1 \oplus \xi_2 = \left( \begin{array}{c} \xi_1 \\ \xi_2 \end{array} \right),
    \end{equation*}
    then since
    \begin{equation*}
        a_1 (\xi \wedge e) = \left( \begin{array}{c} 0 \\ \xi \end{array} \right),
    \end{equation*}
    the lemma can be written as
    \begin{equation*}
        M = - \left( \begin{array}{c} 0 \\ U^E \end{array} \right) \wedge \dbar\frac{1}{w}.
    \end{equation*}
    We will prove by induction on $k$ that
    \begin{equation} \label{eqmexpr}
        M_{k+1} = \left( \begin{array}{c} 0 \\ U_k^F \end{array} \right) \wedge \dbar\frac{1}{w},
    \end{equation}
    for $k \geq 1$. We thus start with the case $k=1$. Then,
    \begin{equation*}
        M_2 = \dbar |F|^{2\lambda} \wedge \sigma_2^G a_1 \sigma_1^F |_{\lambda=0} =
        - \left.\left( \begin{array}{c} 0 \\ \sigma^E_{1,w} \end{array} \right) \dbar |F|^{2\lambda} \frac{1}{w} \right|_{\lambda=0} 
            = - \left( \begin{array}{c} 0 \\ \sigma^E_{1,w} \end{array} \right) \dbar \frac{1}{w}
    \end{equation*}
    outside of $Z_1^E\times \{ 0 \}$, where $\sigma^E_{1,w}$ is smooth.
    Since by \eqref{eqsigmawexpr2}, for $z \notin Z^E_1$, $\sigma^E_{1,w} = \sigma^E_1 + |w|^2\alpha$,
    and since the last term annihilates $\dbar(1/w)$, we thus get that
    \begin{equation*}
        M_2 = - \left( \begin{array}{c} 0 \\ \sigma^E_1 \end{array} \right) \dbar \frac{1}{w}
    \end{equation*}
    outside of $Z^E_1 \times \{ 0 \}$. For $z \notin Z^E_1$, $U^E_1$ is smooth and equals $\sigma^E_1$,
    so outside of $Z^E_1 \times \{ 0 \}$,
    \begin{equation} \label{eqm2expr}
        M_2 = - \left( \begin{array}{c} 0 \\ U^E_1 \end{array} \right) \dbar \frac{1}{w}.
    \end{equation}
    Now, both sides of \eqref{eqm2expr} are pseudomeromorphic $(0,1)$-currents,
    which coincide outside of $Z^E_1 \times \{ 0 \}$, which has codimension $\geq 2$,
    so by the dimension principle, they must coincide everywhere, and we have proven
    \eqref{eqmexpr} for $k = 1$.

    We now begin with the induction procedure, and assume that \eqref{eqmexpr} holds for $k$.
    Then,
    \begin{equation*}
        M_{k+2} = \dbar\sigma^G_{k+2}  M_{k+1}
            = - \dbar\sigma^G_{k+2}\left( \begin{array}{c} 0 \\ U^E_k \end{array} \right) \wedge \dbar \frac{1}{w},
    \end{equation*}
    outside of $Z^E_{k+1}$, where $\sigma^G_{k+2}$ is smooth.
    By \eqref{eqsigmagexpr2},
    \begin{equation*}
- \dbar\sigma^G_{k+2}\left( \begin{array}{c} 0 \\ U^E_k \end{array} \right) \wedge \dbar \frac{1}{w}
         = - \left( \begin{array}{c} 0 \\ \dbar\sigma^E_{k+1,w} U^E_k \end{array} \right) \wedge \dbar \frac{1}{w},
    \end{equation*}
    and as above, we get by \eqref{eqsigmawexpr2} that
    \begin{equation*}
- \left( \begin{array}{c} 0 \\ \dbar\sigma^E_{k+1,w} U^E_k \end{array} \right) \wedge \dbar \frac{1}{w}
         = - \left( \begin{array}{c} 0 \\ \dbar\sigma^E_{k+1} U^E_k \end{array} \right) \wedge \dbar \frac{1}{w},
     \end{equation*}
     outside of $Z^E_{k+1} \times \{ 0 \}$. For $z \notin Z^E_{k+1}$, $\dbar\sigma^E_{k+1}$ is smooth
     and $\dbar\sigma^E_{k+1} U^E_k = U^E_{k+1}$, so to conclude, we get that outside of $Z^E_{k+1} \times \{ 0 \}$,
     \begin{equation} \label{eqmkexpr}
         M_{k+2} 
         = - \left( \begin{array}{c} 0 \\ U^E_{k+1} \end{array} \right) \wedge \dbar \frac{1}{w}.
     \end{equation}
     Now, both sides of \eqref{eqmkexpr} are pseudomeromorphic $(0,k+1)$-currents, which
     coincide outside of $Z^E_{k+1} \times \{ 0 \}$, which has codimension $\geq k+2$,
     so they coincide everywhere by the dimension principle. Thus, we have proven
     by induction that \eqref{eqmexpr} holds for all $k$.
\end{proof}

\section{Currents with prescribed annihilator ideals on arbitrary varieties} \label{sectnonpure}

We will here consider the construction of a current $R^\mathcal{J}_Z$ with annihilator $\mathcal{J}$
on a variety $Z$ as in Section~\ref{sectpure}, but without the assumption of pure dimension,
i.e., $Z$ may consist of irreducible components of different dimensions.

The construction will be essentially the same as in the case of pure dimension,
but one needs to in a certain sense treat the components of $Z$ of different dimension separately,
and thus, treating the case of pure dimension separately should hopefully illustrate
the main ideas better, without needing to delve in to certain technicalities in
the general case.

To begin with, we note that on a variety which is not of pure dimension,
talking about the bidegree of a current does not have any meaning,
while the bidimension (i.e., the bidegree of the test forms it is acting on)
still does.
For example, considering the union $Z$ of a complex line and a complex plane
in $\C^3$, intersecting at the origin, then the integration current $[0]$
is a current on $Z$ of bidimension $(0,0)$. However, if we consider $[0]$ as a
current on the line, it would have bidegree $(1,1)$, while on the plane, it
would have bidegree $(2,2)$.
Note also that the bidimension of a current is preserved under push-forwards
under inclusions (in contrast to the bidegree in the case of pure dimension,
which increases by the codimension under push-forwards).
We will thus in this section need to reformulate statements in terms of bidimension
instead of bidegree of currents. For example, the dimension principle needs to be
formulated in the following natural form.
\begin{prop} \label{proppmdim2}
    If $T \in \PM(Z)$ is a current of bidimension $(c,d)$ with support on a variety $V$,
    and $\dim V < d$, then $T = 0$.
\end{prop}
The proof works the same as in the smooth case, by first proving that
$\overline{h}T = 0$ and $d\overline{h}\wedge T = 0$ if $h$ is a holomorphic
function vanishing on $\supp T$. Then, if $i : Z \to \Omega \subseteq \C^n$
is a local embedding, one proves that $i_* T = 0$ by induction over
$\dim V$, by proving that $i_* T = 0$ on $V_{\rm reg}$ (considered as a subvariety
of $\Omega$).

The rest from Section~\ref{ssectasmpm} about restrictions of pseudomeromorphic
currents, the SEP and almost semi-meromorphic currents works the same as in the
case of pure dimension, as is assumed in \cite{AS2}.
However, one must make sure to interpret the SEP in the right way.
A pseudomeromorphic current $T$ has the SEP with respect to $Z$ if
${\bf 1}_V T = 0$ for all subvarieties $V$ of $Z$ of positive codimension.
By positive codimension, we mean that $\codim V\cap Z_i > \codim Z_i$
for all irreducible components $Z_i$ of $Z$. Note however, that this is not
the same as saying that $\codim V > \codim Z$, which for example any irreducible
component not of maximal dimension would satisfy.

The existence of the structure form $\omega_Z$ takes the following form.

\begin{prop} \label{propomeganonpure}
    Let $(F,\psi)$ be a Hermitian resolution of $\Ok_\Omega/\mathcal{I}_Z$,
    where $Z$ is a subvariety of $\Omega$ of not necessarily pure dimension.
    Let $R^{\mathcal{I}_Z}$ be the associated residue current of $(F,\psi)$,
    and let $W^e$ be the union of the irreducible components of $Z$ of dimension $e$.
    Then there exists an almost semi-meromorphic current
    \begin{equation*}
        \omega_Z = \omega^d + \dots + \omega^0
    \end{equation*}
    on $Z$, where $\dim Z = d$, $\omega^e$ has bidimension
    $(0,e)$, support on $\cup_{f \geq e} W^f$ and takes values in $F_{n-e}$, such that
    \begin{equation} \label{eqomegaRnonpure}
        i_* \omega_Z = R^{\mathcal{I}_Z} \wedge dz,
    \end{equation}
    where $i : Z \to \Omega$ is the inclusion and $dz := dz_1\wedge \dots \wedge dz_n$.
\end{prop}

We can now state and prove the main theorem also in the case when the dimension is not pure,
and will then return to the proof of Proposition~\ref{propomeganonpure}.
The setting will be the same as in Section~\ref{sectpure}, with $\mathcal{J}$ an
ideal in $\Ok_Z$, $\mathcal{\tilde{J}}$ a maximal lifting of the ideal, the morphism $a : (F,\psi) \to (E,\varphi)$
between the free resolutions $(E,\varphi)$ and $(F,\psi)$ of $\Ok_\Omega/\mathcal{\tilde{J}}$ and
$\Ok_\Omega/\mathcal{I}_Z$ respectively. We also let as above, $W^e$ be the union of the irreducible
components of $Z$ of dimension $e$.

Note that when $Z$ does not have pure dimension, we can not use the same definition of $\nu$
as in \eqref{eqnudef} to define smooth forms on $Z_{\rm reg}$. For example, say that $Z$ consists
of components of dimension $d-1$ and $d$, so $W^{d-1} \neq \emptyset$, $W^d \neq \emptyset$.
Let $p = n-d$. Then, $\sigma^E_{p+1}$ is defined and smooth outside of $Z^E_{p+1}$, but
since $W^{d-1} \subset Z_{p+1}^E$ (see Section~\ref{ssectbef}),
$\sigma^E_{p+1}$ is not defined anywhere on $W^{d-1}$, and hence can not be uniquely extended
there. Thus, we alter the definition of $\nu$ to be zero on such components.
We will explain after the statement of the theorem why it does not in fact matter
how we define $\nu$ on such components.

In order to define $\nu$, we note first that $\{ W^e\setminus Z_{\rm sing} \mid e =0,\dots,d \}$
cover disjointly $Z_{\rm reg}$, since any point of $Z$ belongs to at least one $W^e$,
and any point belonging to two irreducible components of different dimensions is necessarily
a singular point. Thus, we get a well-defined smooth form on $Z_{\rm reg}$ by defining it
separately on each $W^e\setminus Z_{\rm sing}$. We then let
\begin{equation*}
    \nu^l_k := \left\{  \begin{array}{cc}
        j^*\sigma_k^E\dbar j^*\sigma_{k-1}^E\dots\dbar j^*\sigma_{l+1}^E & \text{ on } W^e\setminus Z_{\rm sing}, e \geq n-l \\
 0                                                         & \text{ on } W^e\setminus Z_{\rm sing}, e < n-l \end{array} \right.,
\end{equation*}
where $j : W^e \to \Omega$ is the inclusion, and we let
\begin{equation} \label{eqnudef2}
    \nu := \sum_{k > l} \nu^l_k.
\end{equation}

\begin{thm} \label{thmmain2}
    Let $Z \subseteq \Omega \subseteq \C^n$ be an analytic subvariety of $\Omega$ of not necessarily
    pure dimension, where $\Omega$ is an open set in $\C^n$. Let $\mathcal{J} \subseteq \Ok_Z$
    be an ideal.
    Then $\nu$ defined by \eqref{eqnudef2} has an extension as an almost
    semi-meromorphic current to $Z$, which we denote by $V^E$.
    If we let 
    \begin{equation} \label{eqrnonpuredef}
        R^\mathcal{J}_Z  := a \omega_Z - \nabla( V^E \wedge a \omega_Z ),
    \end{equation}
    where $a : (F,\psi) \to (E,\varphi)$ is the morphism in \eqref{eqadef}, then 
    \begin{equation} \label{eqannRJnonpure}
        \ann_{\Ok_Z} R^\mathcal{J}_Z  = \mathcal{J}
    \end{equation}
    and
    \begin{equation} \label{eqRJpushforwardnonpure}
        i_* R^\mathcal{J}_Z = R^{\mathcal{\tilde{J}}}_\Omega,
    \end{equation}
    where $\mathcal{\tilde{J}} \subseteq \Ok_\Omega$ is the maximal lifting of
    $\mathcal{J}$, and $R^{\mathcal{\tilde{J}}}_\Omega$ is the current associated
    to $\mathcal{\tilde{J}}$ as in \eqref{eqRJX}.
\end{thm}

Note that in \eqref{eqrnonpuredef}, we take an extension $V^l_k$ of $\nu^l_k$, in order
to extend $\nu$ from $Z_{\rm reg}$ to the current $V^E$ on $Z$. Then, we consider the
product $V^E \wedge a \omega_Z$. This product consists then of terms $V^{n-e}_k \wedge a_{n-e} \omega^e$,
and thus the behaviour of $V^{n-e}_k$ on $W^f\setminus Z_{\rm sing}$, where $f < e$,
(where we have defined $\nu^{n-e}_k$ to be $0$) will not influence the product,
since $W^f\setminus Z_{\rm sing}$ is disjoint from $\supp \omega^e \subseteq \cup_{g \geq e} W^g$
by Proposition~\ref{propomeganonpure}.

\begin{proof}
    Only minor changes need to be done to the proof of Theorem~\ref{thmmain1} in
    order for it to work in this situation as well.

    First of all, one defines almost semi-meromorphic extensions $\nu^l_k$
    from $W^e\setminus Z_{\rm sing}$ to $W^e$, $e \geq n-l-1$ separately in the same
    way as in \eqref{eqvlkdef}. These almost semi-meromorphic extensions can then be
    further extended by $0$ to almost semi-meromorphic currents on all of $Z$,
    since in any smooth modification of $Z$,
    the irreducible components of $Z$ will split into disjoint manifolds.
    Summing these extensions for fixed $l$ and $k$, we get an almost
    semi-meromorphic current $V^l_k$ on $Z$ extending $\nu^l_k$

    Then, in \eqref{eqmlkomega} and the rest of the proof, $\omega_{l-p}$ is replaced by
    $\omega^{n-l}$, and the equality in \eqref{eqmlkomega} will now follow from \eqref{eqomegaRnonpure}
    instead of \eqref{eqomegaR}, together with the fact that
    $\supp \omega^{n-l} \subseteq \cup_{e \geq n-l} W^e$
    (where we as before assume that $W^e$ consists of the irreducible components of $Z$ of dimension $e$).

    As in the proof of Theorem~\ref{thmmain1}, it is clear that \eqref{eqmlkomega} holds
    on $Z_{\rm reg}$.

    The end of the proof of Theorem~\ref{thmmain1} starts by using that it is clear
    that \eqref{eqmlkomega} holds on $Z_{\rm reg}$ for $p \leq l$, and then it is proved
    that \eqref{eqmlkomega} holds on all of $Z$ by proving simultaneously by induction
    over $k-l$ that $M^l_k$ has the SEP with respect to $Z$, and that $M^l_k = 0$ for $l<k$.

    In the case here, it is instead clear that \eqref{eqmlkomega} holds on $W^e \setminus Z_{\rm sing}$
    for $e \geq n-l$. The proof that \eqref{eqmlkomega} holds on all of $Z$ then follows by
    the same induction argument over $k-l$ as in Theorem~\ref{thmmain1}, but where the induction statement
    that $M^l_k = 0$ for $l < p$ is replaced by that $\supp M^l_k \subseteq \cup_{e \geq n-l} W^e$.
    The base case $k=l$ then follows from Proposition~\ref{propomeganonpure}.
\end{proof}

We now turn to the proof of Proposition~\ref{propomeganonpure}.
Only the case of pure dimension is treated in \cite{AS2}. We will essentially
go through the proof of Proposition~3.3 in \cite{AS2}, and explain how to adapt
the proof to cover also the case when the dimension is not pure.

Since the argument is rather technical, we begin by discussing the main ideas of
the proofs.

If we consider the current $R$ associated to some pure-dimensional ideal $\mathcal{J}$,
of codimension $p$, then $R$ will consist of a ``basic current'' $R_p$ and ``auxiliary currents''
$R_k$ for $k>p$. The reason we call them such is that $R_k$ can be obtained from $R_p$,
by multiplying $R_p$ with some generically smooth form, and extending this as a current
to a current with the SEP on $Z$. For example, $R_{p+1}$ is the standard extension of
$\dbar\sigma_{p+1} \wedge R_p$, which a priori is defined only outside $Z_{p+1}$, and then,
$R_{p+2}$ is the standard extension of $\dbar\sigma_{p+2}\wedge R_{p+1}$, and so on.

The first part of the proof of Proposition~3.3 in \cite{AS2} consists of construction the almost
semi-meromorphic current $\omega_0$ from the ``basic current'' $R_p$, and then,
as a next step, almost semi-meromorphic currents $\omega_k$, $k > 0$ are created from the
``auxiliary currents'' $R_{p+k}$, $k > 0$ by an induction argument.
We will do the same construction on each $W^e$ first, and the ``basic current'' on
$W^e$ will be $R' := {\bf 1}_{W^e} R^{\mathcal{I}_Z}_{n-e} \wedge dz$.
In order to keep this proof to a bit more manageable length, we split the step of
proving that $R'$ is the push-forward of an almost semi-meromorphic current $\tilde{\omega}$
into a separate lemma, Lemma~\ref{lmaomeganonpure}, which then corresponds to the first step
in the proof of Proposition~3.3 in \cite{AS2}.

\begin{lma} \label{lmaomeganonpure}
    Using the notation of Proposition~\ref{propomeganonpure}, let $R' := {\bf 1}_{W^e} R^{\mathcal{I}_Z}_{n-e}\wedge dz$.
    Then, there exists an almost semi-meromorphic current $\tilde{\omega}^e$ on $W^e$ such that
    $j_* \tilde{\omega}^e = R'$, where $j : W^e \to \Omega$ is the inclusion.
\end{lma}

\begin{proof}
    In Proposition~3.3 in \cite{AS2}, $Z$ is assumed to have pure codimension
    $p$. As a preliminary step before the proof of Proposition~3.3 in \cite{AS2}, a vector bundle $G$
    and a morphism $g : G \to F_p$ is defined such
    that $\psi_{p+1} g = 0$, and $g$ has a minimal right-inverse $\sigma_G$,
    defined and smooth outside of $Z_{p+1}$
    (in the notation of \cite{AS2}, $g : F \to E_p$, and $\sigma_G$ is denoted $\sigma_F$).
    We do the same construction for $p = n-e$; it is not essential for this construction
    that $p = \codim Z$ or that $Z$ is of pure dimension.

    The first step in the proof in \cite{AS2} is to define $\omega_0$ on $Z_{\rm reg}$.
    On $W^e\setminus Z_{\rm sing}$, we define $\tilde{\omega}^e$ in the same way as $\omega_0$
    is defined in \cite{AS2}; this definition on the regular part does not rely on $Z$
    being of pure dimension.
    By construction, $i_* \tilde{\omega}^e = R_p^{\mathcal{I}_Z}\wedge dz = R'$ on $W^e\setminus Z_{\rm sing}$.
    We have that $R'$ corresponds to a current on $W^e \setminus Z_{\rm sing}$
    since it is the push-forward of $\tilde{\omega}^e$ there.
    In fact, $R'$ will correspond to a current on all of $W^e$, since if $\phi|_{Z_{\rm reg}} = 0$,
    then $\supp \phi\wedge R' \subseteq Z_{\rm sing} \cap W^e$,
    so
    \begin{equation*}
        \phi\wedge R' = {\bf 1}_{Z_{\rm sing}\cap W^e} (\phi\wedge R') = \phi\wedge {\bf 1}_{Z_{\rm sing}\cap W^e} R' = 0,
    \end{equation*}
    where the last equality holds since $\codim Z_{\rm sing}\cap W^e > \codim W^e = p$,
    and $R'$ is a pseudomeromorphic $(n,p)$-current, so ${\bf 1}_{Z_{\rm sing}\cap W^e} R' = 0$
    by the dimension principle.
    Thus, $\tilde{\omega}^e$ has an extension as a current to $W^e$. If we let $\vartheta = g\tilde{\omega}^e$
    on $W^e\setminus Z_{\rm sing}$, then, as in the equation following (3.19) in \cite{AS2},
    $\dbar \vartheta = 0$ on $W^e\setminus Z_{\rm sing}$ and $\tilde{\omega}^e = \sigma_G\vartheta$.
    In addition, since $\tilde{\omega}^e$ has an extension as a current to $W^e$, so does
    $\vartheta = g\tilde{\omega}^e$, since $g$ is holomorphic (and in particular, smooth).
    By Example~2.8 in \cite{AS2}, $\vartheta$ then has a meromorphic extension to $W^e$.

    As in the proof of Proposition~3.3 in \cite{AS2}, by principalization
    of the Fitting ideal of $g$, followed by a resolution of singularities,
    one gets a smooth modification $\tau : \tilde{Z} \to Z$
    of $Z$ such that the Fitting ideal of $\tau^*g$ is locally principal on $\tilde{Z}$.
    Thus, there exists a line bundle on $\tilde{Z}$ with section $s_G$
    generating this Fitting ideal.
    Then, $\tau^* \sigma_G = \beta_G/s_G$, where $\beta_G$ is smooth.
    We thus get that $j^*\sigma_G$, is almost semi-meromorphic on $W^e$ since it is smooth
    outside of $Z_{p+1}$, which has codimension $\geq p+1$.
    Hence, $\tilde{\omega}^e = \sigma_G\vartheta$ has an extension to $W^e$ as a product
    of almost semi-meromorphic currents and this extension has the SEP with respect
    to $W^e$ by Proposition~\ref{propprodasm}. Since $i_*\tilde{\omega}^e = R'$ on
    $W^e\setminus Z_{\rm sing}$, and both sides have
    extensions over $Z_{\rm sing}$, this equality will hold on all of $W^e$ if we show that
    also $R'$ has the SEP with respect to $W^e$. That $R'$ has the SEP with respect to $Z$
    follows from the dimension principle, since $R'$ is a pseudomeromorphic $(n,p)$-current
    with support on $W^e$ of codimension $p$ (so ${\bf 1}_V R'$ will be a pseudomeromorphic
    $(n,p)$-current with support on $V$ of codimension $\geq 1$ in $W^e$).
\end{proof}

Now, we turn back to the proof of Proposition~\ref{propomeganonpure}, which will correspond to the
last part of the proof of Proposition~3.3 in \cite{AS2}. We write
the current $R_{n-e}$ as $R_{n-e} = R' + R''$, where $R'$ is as above, and $R'' = R_{n-e} - R''$
has a priori support on $\overline{Z\setminus W^e}$. In fact, by the dimension principle,
it will have support on $V := \cup_{d > e} W^d$, i.e., on the irreducible components
of $Z$ if dimension $> e$. There, $R''$ will correspond to the ``auxiliary currents'' above,
created from the ``basic currents'' on each $W^d$ by multiplying with almost semi-meromorphic
forms on $W^d$.

The first construction, with $R'$, works well on $W^e\setminus V$, where $R_{n-e} = R'$,
and the second construction, with $R''$, works well on $V \setminus Z_{n-e}$, where
$\dbar\sigma_{n-e}$ is smooth, and $R_{n-e} = R''$.

The final step of the proof is to treat the parts of $Z$ which the construction
above does not handle, i.e., $W^e \cap (\cup_{d \neq e} W^d)$, $V \cap Z_{n-e}$
and $\cup_{d < e} W^d$. However, we will see that all of these have dimension $< e$
(a key point for this is \eqref{eqzkprim}) so essentially by the dimension principle,
these parts are small enough to not have any influence.

\begin{proof}[Proof of Proposition~\ref{propomeganonpure}]
    For $d = \dim Z$ (i.e., the dimension of the irreducible components of maximal dimension),
    we define $\omega^d := \tilde{\omega}^d$, where $\tilde{\omega}^d$ is from Lemma~\ref{lmaomeganonpure}.
    Since by the dimension principle, $R_p\wedge dz$ has support
    on $W^d$, $R' = {\bf 1}_{W^d} R_p\wedge dz = R_p \wedge dz$. Thus, since
    $i_* \tilde{\omega}^d = R'$, we get that $i_* \omega^d = R_p \wedge dz$, $\omega^d$ is almost
    semi-meromorphic, and $\supp \omega^d \subseteq W^d$.

    As in the proof of Lemma~\ref{lmaomeganonpure}, by principalization
    of the Fitting ideals of $\varphi_k$ for $k \geq \codim Z$, followed by a
    resolution of singularities, one gets a smooth modification $\tau : \tilde{Z} \to Z$
    of $Z$ such that all the Fitting ideals are locally principal on $\tilde{Z}$.
    Thus, there exists line bundles on $\tilde{Z}$ with sections $s_k$
    generating the Fitting ideals of $\tau^* \varphi_k$.
    Then, as in the proof of Proposition~3.3 in \cite{AS2}, $\tau^* \sigma_k = \beta_k/s_k$,
    where $\beta_k$ are smooth, and $\tau^* \dbar\sigma_k = \dbar\beta_k/s_k$.
    We thus get that $i^*\sigma_k$ and $i^*\dbar\sigma_k$ are almost
    semi-meromorphic on the irreducible components of $Z$ where they are generically defined.

    We will now by backwards induction over $e$ define $\omega^e$, such that
    $i_* \omega^e = R_{n-e} \wedge dz$, $\omega^e$ is almost semi-meromorphic and
    $\supp \omega^e \subseteq \bigcup_{f \geq e} W^f$.
    Assume hence that this holds for $\omega^{e+1}$, and let $p = n-e$.
    On $\supp \omega^{e+1} \subseteq \bigcup_{f \geq e+1} W^f =: V$, we have that
    $j^* \dbar\sigma_p$ is almost semi-meromorphic, where $j : V \to \Omega$ is
    the inclusion, since it is generically defined outside of $Z_p$, which has dimension $\leq e$.
    Then, we let
    \begin{equation*}
        \omega^e := \tilde{\omega}^e + j^*(\dbar\sigma_p) \omega^{e+1}.
    \end{equation*}
    Since $i_* \omega^{e+1} = R_{p-1} \wedge dz$, and $R_p = \dbar\sigma_p R_{p-1}$
    outside of $Z_p$, we get that $i_* \omega^e = R^p \wedge dz$ outside of $Z_p$.
    In addition, we have that $i_* \omega^e = R^p \wedge dz$ on $W^e\setminus Z_{\rm sing}$
    by construction of $\tilde{\omega}^e$ and the fact that $\omega^{e+1}$ has no support there.
    In conclusion, $i_* \omega^e = R_p \wedge dz$ outside of $(W^e\cap Z_{\rm sing}) \cup (V\cap Z_p)$.
    Both sides have current extensions over this set, and $\omega^e$ being almost semi-meromorphic
    thus has the SEP on $W^e \cup V$. It thus remains to see that also $R_p \wedge dz$ has the
    SEP in order to finish the induction step. This will hold by the dimension principle,
    since $R^p\wedge dz$ is of bidegree $(n,p)$, and $\dim ((W^e\cap Z_{\rm sing}) \cup (V\cap Z_p)) < e$.
    To see this last part, we note first that
    $W^e \cap Z_{\rm sing} = W^e_{\rm sing}\cup (W^e\cap (\cup_{f\neq e} W^f))$, of which both
    of the sets in this union have codimension $\geq 1$ in $W^e$.
    In addition, by \eqref{eqzkprim}, $\dim V\cap Z_p < e$.
\end{proof}

We consider an example of such a structure form. The calculation becomes rather involved,
even though this is probably the simplest case of a variety which is not of pure dimension.

\begin{ex}
    Let $Z = \{ x = y = 0 \} \cup \{ z = 0 \} = Z(xz,yz) \subseteq \C^3$.
    Then $\Ok/\mathcal{I}_Z$ has a free resolution
    \begin{equation*}
        0 \to \Ok \stackrel{\varphi_2}{\to} \Ok^{\oplus 2} \stackrel{\varphi_1}{\to} \Ok \to \Ok/\mathcal{I}_Z,
    \end{equation*}
    where
    \begin{equation*}
        \varphi_2 = \left( \begin{array}{c} -y \\ x \end{array} \right) \text{ and }
        \varphi_1 = \left( \begin{array}{cc} xz & yz \end{array} \right),
    \end{equation*}
    i.e., it is like the Koszul complex of $(x,y)$, except for the factors $z$ of the
    entries in $\varphi_1$.
    We first compute the current $R^E$ associated to this free resolution.
    Since $R^E$ has support on $Z$, by the dimension principle, we get that
    $R^E_1$ has support on $\{ z = 0 \}$. Looking first on $\{ z = 0 \} \setminus \{ x = 0 \}$,
    $\mathcal{I}_Z$ is generated by $z$. Applying the comparison formula to $(E,\varphi)$,
    and the free resolution $(F,\psi)$ of $\Ok/\mathcal{J}(z)$, where $F_1 \cong F_0 \cong \Ok$
    and $\psi_1 = z$, we get that the morphism $a : (F,\psi) \to (E,\varphi)$ becomes
    $a_1 = (\begin{array}{cc} 1/x & 0\end{array})^t$.
    Since the current associated to $F$ equals $\dbar(1/z)$, we get by \eqref{eqRcomparison} and
    \eqref{eqmdef} that
    $R^E_1 = (I_{E_1} - \varphi_2\sigma^E_2)(\begin{array}{cc}1/x & 0\end{array})^t \dbar(1/z)$.
    Using that $\sigma^E_2 = (\begin{array}{cc} -\bar{y} & x \end{array})/(|x|^2+|y|^2)$, we get
    that outside of $\{ x = z = 0 \}$,
    \begin{equation} \label{eqre1}
        R^E_1 = \frac{1}{|x|^2+|y|^2}\left(\begin{array}{c} \bar{x}\\\bar{y} \end{array}\right)\dbar\frac{1}{z}.
    \end{equation}
    By the dimension principle, this holds everywhere, since $R^E_1$ is a pseudomeromorphic $(0,1)$-current,
    and $\codim \{ x = z = 0 \} = 2$.
    Regarding what this means at $\{ 0 \}$, cf., the discussion of standard extensions in
    Example~5 in \cite{Lar3}.

    Outside $\{ z = 0 \}$, then $\mathcal{I}_Z = (x,y)$, and the free resolution $(E,\varphi)$ of $\Ok/\mathcal{I}_Z$
    will differ from the Koszul complex of $(x,y)$ only by the factor $z$ in the entries $\varphi_1$.
    This will cause an extra factor $1/z$ in $\sigma^E_1$ compared to the $\sigma_1$ associated to the
    Koszul complex. Since the current associated to the Koszul complex of $(x,y)$ is $\dbar(1/y)\wedge \dbar(1/x)$,
    we get that
    \begin{equation} \label{eqre2}
        R^E_2 = \frac{1}{z} \dbar\frac{1}{y}\wedge \dbar\frac{1}{x}
    \end{equation}
    outside of $\{ z = 0 \}$. On the other hand, since $Z^E_2 = \{ x = y = 0 \}$, we have outside of $Z^E_2$
    that $R^E_2 = \dbar\sigma^E_2 R^E_1$, and combining this with \eqref{eqre1} and \eqref{eqre2}, we get that
    \begin{align*}
        R^E = \frac{1}{|x|^2+|y|^2}\left(\begin{array}{c} \bar{x}\\\bar{y} \end{array}\right)\dbar\frac{1}{z}
            + \frac{1}{z} \dbar\frac{1}{y}\wedge \dbar\frac{1}{x} +  \\
            \dbar\left(\frac{\left(\begin{array}{cc} -\bar{y} & \bar{x} \end{array}\right)}{|x|^2+|y|^2}\right)
                \frac{1}{|x|^2+|y|^2}\left(\begin{array}{c} \bar{x}\\\bar{y} \end{array}\right)\wedge\dbar\frac{1}{z}
    \end{align*}
    outside of $\{ x = y = 0 \} \cap \{ z = 0 \} = \{ 0 \}$.
    By the dimension principle, this thus holds everywhere, since the components of $R^E$ are of either
    bidegree $(0,1)$ or $(0,2)$ and $\codim \{ 0 \} = 3$.
    Taking the wedge product with $\omega_{\C^3} = dx\wedge dy\wedge dz$, and using that
    $\dbar(1/y)\wedge\dbar(1/x)\wedge dx\wedge dy = (2\pi i)^2 [x = y = 0]$ and $\dbar(1/z) \wedge dz = 2\pi i[z = 0]$,
    we get by \eqref{eqomegaRnonpure} that
    \begin{align*}
        \omega_Z = (2\pi i)^2 \chi_{\{x = y = 0 \}} \frac{dz}{z} +
        2\pi i \chi_{\{z = 0 \}} \left(
        \frac{dx\wedge dy}{|x|^2+|y|^2}\left(\begin{array}{c} \bar{x}\\\bar{y} \end{array}\right) + \right.\\
        \left.\dbar\left(\frac{\left(\begin{array}{cc} -\bar{y} & \bar{x} \end{array}\right)}{|x|^2+|y|^2}\right)
                \frac{dx\wedge dy}{|x|^2+|y|^2}\left(\begin{array}{c} \bar{x}\\\bar{y} \end{array}\right)\right),
    \end{align*}
    where $\chi_{\{x = y = 0 \}}$ and $\chi_{\{z = 0 \}}$ are the characteristic functions for the respective
    zero sets.
\end{ex}

\section{Free resolutions on singular varieties} \label{sectnaive}

Given an ideal $\mathcal{J}\subseteq \Ok_Z$, where $Z \subseteq \Omega$,
the construction of the current $R^\mathcal{J}_Z$ relied on free resolutions
over $\Ok_\Omega$ of the maximal lifting $\mathcal{\tilde{J}}$ of $\mathcal{J}$.
A more natural generalization of the construction in \cite{AW1} would be to
consider free resolutions intrinsically
on $Z$, i.e., a free resolution of $\Ok_Z/\mathcal{J}$ over $\Ok_Z$, which
(at least locally) exists also on a singular variety.
We discuss in this section why this approach does not work.

One of the differences between free resolutions of ideals in the smooth and singular
case is that the free resolutions need not be of finite length in the latter case,
see Example~\ref{exinfres} below for an example of this.
In fact, a famous result by Auslander-Buchsbaum-Serre states that a Noetherian
local ring $R$ is regular if and only if all finitely generated $R$-modules
have free resolutions of finite length.
If $R = \Ok_{Z,z}$, then $R$ is regular if and only if $z$ is a regular point of $Z$.
However, even when the ideals do have finite free resolutions, the construction
of Andersson and Wulcan will in general not have the correct annihilator.
This is essentially treated in \cite{Lar2}, but we will elaborate a bit here how
this applies to our situation. We consider first an example, where
one can get an indication of what can go wrong.

\begin{ex}
    Let, as in Section~\ref{ssectpihyper}, $Z$ be a reduced hypersurface defined
    by a holomorphic function $h$, and let $f$ be a non-zero-divisor in $\Ok_Z$.
    Note that $f$ being a non-zero-divisor means precisely that the complex
    $\Ok_Z \stackrel{(f)}{\to} \Ok_Z$ is a free resolution of $\Ok_Z/\mathcal{J}(f)$.
    Hence, the current associated to this free resolution is the residue
    current $\dbar(1/f)$.
    Consider the push-forward of $\dbar(1/f)$ to 
    the ambient space, $i_* \dbar(1/f) = \dbar(1/\tilde{f}) \wedge [Z]$, where
    $\tilde{f}$ is a representative of $f$ in $\Omega$.
    By the Poincar\'e-Lelong formula, see \cite{CH}, Section~3.6,
    \begin{equation*}
        \dbar\frac{1}{\tilde{f}}\wedge [Z] =
        \frac{1}{2\pi i} \dbar \frac{1}{\tilde{f}}\wedge\dbar\frac{1}{h}\wedge dh.
    \end{equation*}
    Now, if $\phi \dbar(1/\tilde{f})\wedge[Z] = 0$, then the coefficients of $\phi dh$
    lie in $\mathcal{J}(\tilde{f},h)$ by the duality theorem. However, since $dh$ vanishes
    on $Z_{\rm sing}$, this does not necessarily imply that $\phi \in \mathcal{J}(\tilde{f},h)$.
    Indeed, we show in \cite{Lar2} that if $\codim Z_{\rm sing} = 1$, then one can find
    $\phi$ and $f$ such that $\phi dh \in \mathcal{J}(\tilde{f},h)$ but $\phi \notin \mathcal{J}(\tilde{f},h)$.
    In that case, we thus get that $\ann \dbar(1/f) \neq \mathcal{J}(f)$.
\end{ex}

We now turn to the general case.
Consider a singular subvariety $Z \subseteq \Omega$ of codimension $p$.
Let $Z^0 := Z_{\rm sing}$ and $Z^k := Z_{k+p}$ for $k \geq 1$,
where $Z_{k+p}$ are the singularity subvarieties associated to a free resolution
of $\Ok_Z$. Let $q$ be the largest integer such that $\codim Z^k \geq k+q$,
where by $\codim Z^k$, we refer to the codimension of $Z^k$ in $Z$.
(Since $Z$ is assumed to be singular, $Z^0 = Z_{\rm sing} \neq \emptyset$, and hence, $q \leq \dim Z$.)
By Corollary~9 in \cite{Lar2} there exists a complete intersection $f = (f_1,\dots,f_q)$ on $Z$ such that
$\ann \mu^f \neq \mathcal{J}(f)$. By Theorem~\ref{thmbmch}, $\mu^f$ equals the Bochner-Martinelli
current of $f$, i.e., the current associated to the Koszul complex of $f$.
We claim that in this case, the Koszul complex of $f$ is a free resolution of $\mathcal{J}(f)$,
and hence what we described above show that the naive generalization of the construction by
Andersson and Wulcan does not work in this case.
To see that the Koszul complex of $f$ is exact, we note first that by Theorem~4 in \cite{Lar2},
if $f' = (f_1,\dots,f_{q'})$, where $q' < q$, then $\ann \mu^{f'} = \mathcal{J}(f')$,
and by Lemma~30 in \cite{Lar2}, $(f_1,\dots,f_q)$ is then a regular sequence.
By \cite{Eis}, Corollary~17.5, the Koszul complex of $f$ is then a free resolution
of $\Ok_Z/\mathcal{J}(f)$.

We saw however in Example~\ref{expurecm} that if $Z$ is Cohen-Macaulay, there was an
easy remedy for this, we should consider $R = \omega - \nabla(U^f\omega)$ instead of $I-\nabla(U^f)$.
If $Z$ is not Cohen-Macaulay, or if we have an ideal which does not lift to a Cohen-Macaulay ideal,
it is not as clear how to remedy this.

We consider also another issue arising when the free resolutions on the variety are not of finite length.
\begin{ex} \label{exinfres}
    Let $Z = \{ x = 0 \} \cup \{ y = 0 \} = \{ xy = 0 \} \subseteq \C^2$.
    Consider the ideal $\mathcal{J} = \mathcal{J}(x) \subseteq \Ok_Z$.
    It is easily verified that if $E_k \cong \Ok_Z$, $\varphi_{2k+1} = (x)$,
    $\varphi_{2k+2} = (y)$, $k = 0,1,\cdots$, then $(E,\varphi)$ is a free resolution of $\Ok_Z/\mathcal{J}(x)$
    over $\Ok_Z$.
    In addition, since $x\in \mathfrak{m}$ and $y \in \mathfrak{m}$, where
    $\mathfrak{m} := \mathcal{J}(x,y)$ is the maximal ideal in $\Ok_{Z,0}$,
    we have that $(E,\varphi)$ is a minimal free resolution over the local ring $\Ok_{Z,0}$, see
    \cite{Eis}, Theorem~20.2. This theorem about uniqueness of minimal free resolutions
    holds over any Noetherian local ring, without any requirements about regularity
    of the ring, so since $(E,\varphi)$ is one minimal free resolution of $\Ok_Z/\mathcal{J}$
    over $\Ok_Z$ of infinite length, any other free resolution must also be of infinite length.

    We now consider the sets $Z^E_k$, where $\varphi_k$ does not have minimal rank.
    They are $Z^E_{2k+1} = \{ x = 0 \}$ and $Z^E_{2k+2} = \{ y = 0 \}$, $k = 0,1,\dots$.
    Note that $\codim Z^E_k = 0$ and $Z^E_{2k+2} \not\subseteq Z(\mathcal{J})$.
    This shows that the Buchsbaum-Eisenbud criterion and its corollaries, as described in
    Section~\ref{ssectbef}, fail. The reason for this is not directly that the ring we consider
    is not regular, the Buchsbaum-Eisenbud criterion holds on any Noetherian local ring.
    However, the criterion does not apply here, since one requirement is that the complex
    is of finite length.
    Since much of the construction of Andersson-Wulcan currents relies on the Buchsbaum-Eisenbud
    criterion and its corollaries, this would be an obstacle to overcome in order to construct
    such currents directly from free resolutions on the variety, without going to a lifting
    of the ideal as we do in this article.
\end{ex}

\end{document}